\documentclass[11pt]{amsart}

\usepackage{color}
\usepackage{amsmath}
\usepackage{amsxtra}
\usepackage{amscd}
\usepackage{amsthm}
\usepackage{amsfonts}
\usepackage{amssymb}
\usepackage{eucal}
\usepackage{epsfig}
\usepackage{graphics}
\usepackage[matrix,arrow]{xy}

\usepackage{hyperref}

\theoremstyle{plain}
\newtheorem{theorem}{Theorem}[section]

\newtheorem{thm}[theorem]{Theorem}
\newtheorem{cor}[theorem]{Corollary}
\newtheorem{prop}[theorem]{Proposition}
\newtheorem{lem}[theorem]{Lemma}

\newtheorem{defi}[theorem]{Definition}
\newtheorem{example}[theorem]{Example}

\newcommand{\Glie}{\mathfrak{g}}
\newcommand{\Yim}{\mathcal{Y}}

\newcommand{\U}{\mathcal{U}}

\newcommand{\ZZ}{\mathbb{Z}}
\newcommand{\CC}{\mathbb{C}}

\newcommand{\C}{\mathbb{C}}

\newcommand{\Z}{\mathbb{Z}}

\newcommand{\g}{\mathfrak{g}}

\newcommand{\Psib}{\mbox{\boldmath$\Psi$}}

\newcommand{\nc}{\newcommand}
\nc{\on}{\operatorname}
\nc{\la}{\lambda}
\nc{\wh}{\widehat}
\nc{\wt}{\widetilde}
\nc{\sw}{{\mathfrak s}{\mathfrak l}}
\nc{\ghat}{\wh{\g}}
\nc{\hhat}{\wh{\h}}
\nc{\mc}{\mathcal}
\nc{\bi}{\bibitem}
\nc{\pa}{\partial}
\nc{\ppart}{(\!(t)\!)}
\nc{\pparl}{(\!(\la)\!)}
\nc{\zpart}{(\!(z^{-1})\!)}
\nc{\n}{{\mathfrak n}}
\nc{\ol}{\overline}
\nc{\mb}{\mathbf}
\nc{\bb}{{\mathfrak b}}
\nc{\su}{\wh\sw_2}
\nc{\h}{{\mathfrak h}}
\nc{\can}{\on{can}}
\nc{\ntil}{\wt{\n}}
\nc{\pone}{{\mathbb P}^1}
\nc{\bs}{\backslash}
\nc{\al}{\alpha}
\nc{\gt}{{\mathfrak g}'}
\nc{\ds}{\displaystyle}

\nc{\Sig}{{\mb \Sigma}}

\theoremstyle{definition}
\newtheorem{rem}[theorem]{Remark}

\begin{document}

\begin{title}
{Weyl group symmetry of $q$-characters}
\end{title}

\author{Edward Frenkel}

\address{Department of Mathematics, University of California,
  Berkeley, CA 94720, USA}

\author[David Hernandez]{David Hernandez}

\address{Universit\'e Paris Cit\'e and Sorbonne Universit\'e, CNRS,
  IMJ-PRG, IUF, F-75006, Paris, France}

\begin{abstract} 
 We define an action of the Weyl group $W$ of a simple Lie algebra
 $\g$ on a completion of the ring ${\mc Y}$, which is the codomain of
 the $q$-character homomorphism of the corresponding quantum affine
 algebra $U_q(\ghat)$. We prove that the subring of $W$-invariants
 of ${\mc Y}$ is precisely the ring of $q$-characters,
 which is isomorphic to the Grothendieck ring of the category of
 finite-dimensional representations of $U_q(\ghat)$. This resolves an
 old puzzle in the theory of $q$-characters. We also identify the
 screening operators, which were previously used to describe the ring
 of $q$-characters, as the subleading terms of simple reflections from
 $W$ in a certain limit. Our results have already found applications
 to the study of the category ${\mc O}$ of representations of the
 Borel subalgebra of $U_q(\ghat)$ in \cite{FH3} and to the
 categorification of cluster algebras in \cite{GHL}.
\end{abstract}

\maketitle

\tableofcontents

\section{Introduction}

Let $\g$ be a simple Lie algebra over $\C$ of rank $n$ and $G$ the
corresponding simply-connected Lie group. Let $\on{Rep} G$ be the
Grothendieck ring of finite-dimensional representations of $G$
(equivalently, its Lie algebra $\g$). Denote by $T$ the Cartan
subgroup of $G$. Attaching to each finite-dimensional representation
$V$ of $G$ its character, i.e. the function $\chi_V: T \to \C$ defined
by $\chi_V(t) = \on{Tr}_V(t), t \in T$, we obtain an injective
homomorphism of commutative algebras $$\chi: \on{Rep} G \to \Z[T]
\simeq \Z[y_i^{\pm 1}]_{i \in I}.$$ Here $I$ denotes the set
$\{1,2,\ldots,n \}$ and $y_i$ is the $i$th fundamental
weight. Together, these fundamental weights generate the lattice of
homomorphisms $T \to \C^\times$. Moreover, it is well-known that the
image of $\chi$ is isomorphic to the subring of invariants of
$\Z[y_i^{\pm 1}]_{i \in I}$ under the action of the {\em Weyl group}
$W$ of $G$. This group is generated by the simple reflections $s_i$
acting on the fundamental weights by the formula
\begin{equation}    \label{si}
  s_i \cdot y_j := y_j a_i^{-\delta_{ij}},
\end{equation}
where $a_i$ is the monomial corresponding to the $i$th simple root,
\begin{equation}    \label{ai}
  a_i := \prod_{j \in I} y_j^{C_{ji}},
\end{equation}
$C=(C_{ij})$ is the Cartan matrix of $\g$, and $\delta_{ij}$ is the
Kronecker delta.

In \cite{Fre}, N. Reshetikhin and one of the authors introduced the
notion of {\em $q$-characters}. These are the analogues of characters
for finite-dimensional representations of the level 0
  quotient $\U_q(\ghat)$ of the quantum affine algebra associated to
the affine Kac--Moody algebra $\ghat$ which is the affinization of
$\g$. We will assume throughout this paper that $q \in \C^\times$ is not
a root of unity. The role of $\Z[y_i^{\pm 1}]_{i \in I}$ is now played
by the ring of Laurent polynomials
\begin{equation}    \label{YY}
  \Yim := \ZZ[Y_{i,a}^{\pm 1}]_{i\in I,a\in\CC^\times}
\end{equation}
and the role of $\chi$ is played by the injective $q$-character
homomorphism
\begin{equation}    \label{chiq}
\chi_q: \text{Rep}(\U_q(\ghat)) \rightarrow \Yim,
\end{equation}
where $\text{Rep}(\U_q(\ghat))$ is the Grothendieck ring of the
category of finite-dimensional (type 1) representations
of $\U_q(\ghat)$ (see Section \ref{screenings}) and
$$
\chi_q(V) := \sum_m \text{dim}(V_m) \cdot m.
$$
Here the sum is over the set of all monomials $m$ in $Y_{i,a}^{\pm
  1}$, which are in one-to-one correspondence with the set of
$\ell$-weights of $\U_q(\ghat)$, and $V_m$ is the corresponding
$\ell$-weight subspace of $V$. We will call the image of $\chi_q$ the
    {\em ring of $q$-characters}. It follows from the injectivity of
    $\chi_q$ that it is isomorphic to
    $\text{Rep}(\U_q(\ghat))$.

The theory of $q$-characters has proved to be very useful in
representation theory of quantum affine algebras, and it also plays an
important role in other areas, such quantum integrable models, quiver
varieties, and cluster algebras.

It is natural to ask whether there is an action of the Weyl group $W$
on $\Yim$, so that the subring of $W$-invariants is equal to the ring
of $q$-characters. In fact, natural $q$-analogues of the monomials
$a_i$ given by formula \eqref{ai} are known \cite{Fre}; these are the
monomials
\begin{multline}    \label{Ai}
  A_{i,a} := \\ Y_{i,aq_i^{-1}}Y_{i,aq_i}
\left( \prod_{\{j\in I \mid C_{j,i} = -1\}}Y_{j,a}
\prod_{\{j\in I \mid C_{j,i} = -2\}}Y_{j,aq^{-1}} Y_{j,aq}
\prod_{\{j\in I \mid C_{j,i} =
-3\}}Y_{j,aq^{-2}} Y_{j,a} Y_{j,aq^2} \right)^{-1},
\end{multline}
where $q_i = q^{d_i}$ and the $d_i$'s are the relatively prime
positive integers such that the matrix $\on{diag}[d_1,\ldots,d_n]
\cdot C$ is symmetric. Using this formula,
  V. Chari \cite{C} defined $q$-analogues of the simple reflections
  $s_i$ given by formula \eqref{si} as the automorphisms of $\Yim$
  acting by (up to replacing $q$ by $q^{-1}$)
\begin{equation}    \label{Ti}
  T_i \cdot Y_{j,a} := Y_{j,a}A_{i,aq_i^{-1}}^{-\delta_{i,j}}.
\end{equation}
(closely related operators were also defined in \cite{BP} in the
framework of deformed ${\mc W}$-algebras). However, as was shown
in \cite{C}, the automorphisms $T_i$ generate the braid group
associated to $\g$, {\em not} the Weyl group (in fact, each $T_i$ has
infinite order, whereas the simple reflections $s_i \in W$ have order
$2$). Furthermore, it is easy to see that the subring of
$T_i$-invariants of $\Yim$ is equal to $\Z$ (i.e. consists of the
constant elements).

On the other hand, in \cite{Fre} another collection of operators,
$S_i, i \in I$, called the {\em screening operators}, was
introduced. This was motivated by the close relation between the ring
of $q$-characters and the deformed ${\mc W}$-algebra associated to
$\g$. It was conjectured in \cite{Fre} and proved in \cite{Fre2} that
the ring of $q$-characters coincides with the subring of invariants of
the screening operators, so in a sense, they may be viewed as a
replacement for the action of the simple reflections $s_i, i \in I$,
and hence of the Weyl group, on $\Z[y_i^{\pm 1}]_{i \in I}$.

However, unlike simple reflection $s_i$, which are automorphisms of
$\Z[y_i^{\pm 1}]_{i \in I}$, each screening operator $S_i$ is a {\em
  derivation} acting from $\Yim$ to a free module over $\Yim$, and
this creates a puzzling disparity between the theories of characters
and $q$-characters. Nonetheless, it was generally believed that this
is the best one can do.

In the present paper, we resolve this puzzle and introduce a genuine
action of the Weyl group $W$ such that the ring of $q$-characters gets
identified with the subring of $W$-invariants in $\Yim$. Moreover, we
show that the screening operator $S_i$ naturally arises as a
subleading term in a certain limit of a natural one-parameter
deformation of (a component of) the automorphism $\Theta_i, i \in I$,
corresponding to the action of the simple reflection $s_i \in W$.
This explains why $S_i$ is a derivation. In
addition, we show that Chari's automorphisms $T_i$ generating the
braid group also appear as the leading terms of the $\Theta_i$'s
in a different sense.

An important aspect of our construction is that the automorphisms
$\Theta_i$ involve infinite series. Hence they do not
  preserve $\Yim$ but rather map elements of $\Yim$ to a
  direct sum $\Pi$ of its completions. Our first main result is that
these automorphisms extend to well-defined automorphisms of the ring
$\Pi$ where they generate an {\em action of the Weyl group} $W$. Our
second main result is that under this action, {\em the subring of
  $W$-invariants in $\Yim$ is equal to} $\text{Rep}(\U_q(\ghat))$
(i.e. the ring of $q$-characters of $\U_q(\ghat)$).

Explicitly, the automorphism $\Theta_i$ is given by the formulas $\Theta_i
\cdot Y_{j,a}^{\pm 1} := Y_{j,a}^{\pm 1}$, if $j\neq i$, and
\begin{equation}    \label{Thetai}
  \Theta_i \cdot Y_{i,a} := Y_{i,a} A_{i,aq_i^{-1}}^{-1} \;
  \frac{\Sigma_{i,aq_i^{-3}}}{\Sigma_{i,aq_i^{-1}}},
\end{equation}
where $\Sigma_{i,a}$ is a unique solution of the
$q$-difference equation
\begin{equation}\label{firstqe}
\Sigma_{i,a} = 1 + A_{i,a}^{-1}
\Sigma_{i,aq_i^{-2}}
\end{equation}
in the above completion $\Pi$. Note that in the limit $q
  \to 1$, the numerator and the denominator in formula \eqref{Thetai}
  cancel out and \eqref{Thetai} becomes \eqref{si}.

The element $\Sigma_{i,a}$ can be expanded as a formal series in two
ways. The first expansion is
\begin{equation}    \label{Sigmai}
  \sum_{k\geq 0} \; \prod_{0 < j \leq k}
  A_{i,aq_i^{-2j+2}}^{-1} = 1 + A_{i,a}^{-1} (1 + A_{i,aq_i^{-2}}^{-1}(1 +
  \ldots )),
\end{equation}
where the $k=0$ term in the summation is defined to be $1$. The second
expansion is
\begin{equation}    \label{Sigmaibis}
  - \sum_{k > 0} \; \prod_{0 < j \leq k}
  A_{i,aq_i^{2j}} = - A_{i,aq_i^2} (1 + A_{i,aq_i^4}(1 +
  \ldots )).
\end{equation}
Considering simultaneously different expansions of solutions of
$q$-difference equations such as \eqref{firstqe} is a crucial
ingredient of our construction of the Weyl group action (see
Definition \ref{exsigor} and Section \ref{Weylaction}).

There is an analogous phenomenon in the case of the ring of characters
of finite-dimensional representations of $\g$. Namely, suppose we want
to adjoin to this ring the characters of infinite-dimensional
representations of $\g$ from the category ${\mc O}$. Recall that the
objects of this category have finite-dimensional weight subspaces,
whose weights belong to a subset of the form
\begin{equation}    \label{charO}
\bigcup_{j = 1,\ldots,N} \{ \lambda_j - \sum_{i
  \in I} n_i \al_i \mid n_i \geq 0 \}.
\end{equation}
Hence, the characters of representations from the category ${\mc O}$
belong to a completion of the ring $\Z[y_i^{\pm 1}]_{i \in I}$. But
the problem is that the action of the Weyl group $W$ on $\Z[y_i^{\pm
    1}]_{i \in I}$ does not extend to this completion. However, we can
rectify this situation by introducing a similar category ${\mc O}_w$
for each element $w \in W$. Its objects have non-zero weight subspaces
only for the weights from a subset whose image under $w$ is of the
form \eqref{charO}. The characters of representations from the
category ${\mc O}_w$ belong to another completion of $\Z[y_i^{\pm
    1}]_{i \in I}$, and it is easy to see that the action of $W$ on
$\Z[y_i^{\pm 1}]_{i \in I}$ naturally extends to the direct sum $\pi$
of these completions for all $w \in W$.

In the limit $q \to 1$ our direct sum $\Pi$ of
  completions of $\Yim$ becomes $\pi$, and the automorphisms
$\Theta_i$ reduce to the simple reflections $s_i \in W$ acting on
$\pi$. Moreover, just like $\pi$, our $\Pi$ has an
interpretation in terms of characters of representations from an
analogue of the category ${\mc O}$ for the quantum affine algebra
$\U_q(\ghat)$ (and its twists by $w \in W$).

More precisely, this is the category ${\mc O}$ of representations of
the subalgebra $\U_q(\wh{\mathfrak{b}}) \subset \U_q(\ghat)$, where
$\wh{\mathfrak{b}}$ is a Borel subalgebra of $\ghat$, which was
introduced by M. Jimbo and one of the authors in \cite{HJ}. It
contains all finite-dimensional representations of $\U_q(\ghat)$ as
well as many infinite-dimensional representations whose ordinary
weight subspaces are finite-dimensional and their weights belong to a
subset of the form \eqref{charO}. The $q$-character homomorphism
$\chi_q$ given by formula \eqref{chiq} can be extended from
$\text{Rep}(\U_q(\ghat))$ to the Grothendieck ring $K_0({\mc O})$ of
this category if we also enlarge its target, $\Yim$.

Our definition of the completion $\Pi$, explicit formulas
\eqref{Thetai} for the automorphisms $\Theta_i$, and the statement
that these automorphisms generate the Weyl group $W$ were motivated by
our investigation of some novel relations in $K_0({\mc O})$
(which can also be viewed as relations between $q$-characters of
representations from the category ${\mc O}$). These are the
  extensions of the generalized Baxter $TQ$-relations and the
  $QQ$-system established in our earlier works \cite{FH} and
  \cite{FH2}, respectively. (We note that this extended $QQ$-system
  was introduced and studied in \cite{MRV1,MRV2} in the context of
  affine opers. Its Yangian version for simply-laced $\g$ was
  introduced in \cite{MV1,MV2}, and the corresponding extended version
  was studied in \cite{ESV,EV}.) Both the extended $TQ$-relations and
  the extended $QQ$-systems are labeled by elements of the Weyl group
  $W$. We have conjectured them (and proved in some cases) in our
  recent work \cite{FH3}, in which we have used the Weyl group action
  defined in the present paper as a guiding principle.

In addition, our Weyl action plays an important role in the definition given
in \cite{GHL} of a new cluster algebra structure on the Grothendieck
rings of representations of the so-called shifted quantum affine
algebras. This is another application of the results of the present
paper.

After we obtained the results of the present paper, we learned of the
paper by R. Inoue \cite{In} in which operators similar to one of the
components of our $\Theta_i$ were constructed. Inoue showed that in
the case that $q$ is a {\em root of unity}, they satisfy the relations
of the Weyl group $W$ and that the ring of $q$-characters is contained
in the subring of $W$-invariants (see also \cite{IY}). For generic
$q$, in Remark 4.15 of \cite{In}, Inoue asked whether these operators
could in some sense satisfy the relations of the Weyl group and
whether there is any connection between them and the screening
operators. In the present paper we give affirmative answers to both
questions, and we also prove that the subring of $W$-invariant
elements of $\Yim$ is exactly the ring of $q$-characters (we note that
$\Pi$, the direct sum of completions of $\Yim$ discussed above which
is crucial in our definition of the Weyl group action, does not appear
in \cite{In}).

\medskip

The paper is organized as follows. In Section \ref{cp} we introduce
completions of the ring $\Yim$ labeled by $w \in W$ which we will need
in order to define the Weyl group action. In Section \ref{Weylaction}
we introduce automorphisms $\Theta_i$ on the direct sum $\Pi$ of these
completions. Our first main result, Theorem \ref{weylaction}, is that
these operators generate an action of the Weyl group $W$ on $\Pi$. The
proof consists of demonstrating the relations $\Theta_i^2 = \on{Id}$
(this is Proposition \ref{invol}) and the braid group relations, which
are proved in Section \ref{braidrel} by reducing the statement to the
case of rank 2 simple Lie algebras (see Theorem
  \ref{weylaction1}). Our proof for rank 2 Lie algebras given in
  Section \ref{gtb} relies on a combinatorial study of the
  $TQ$-relations for those Lie algebras (established in \cite{FH}) at the
  level of $q$-characters; note that we do not use any information about
  representation theory. In Section \ref{sltb}, we also give
  a more direct proof for simply-laced $\g$ that does not use the
  $TQ$-relations. In Section
\ref{screenings} we prove our second main result, Theorem
\ref{invariants}, that the subring of $W$-invariants in $\Yim \subset
\Pi$ is equal to the subring $\text{Rep}(\U_q(\ghat))$ of
$q$-characters of $\U_q(\ghat)$. In order to do that, we relate the
operators $\Theta_i$ to the screening operators. In Section
\ref{othsym} we study the relations between our $W$-action and other
known symmetries. In particular, we introduce a $q$-analogue of the
ring of rational functions equipped with an action of $W$ that
naturally appears in representation theory of $\g$.  In Section
\ref{altpf}, we present 
explicit formulas for some elements of $\Pi$ that can be used to
give an alternative, purely combinatorial, proof of the braid
relations.

\medskip

{\bf Acknowledgments.} The authors were partially supported by a grant
from the France-Berkeley Fund of UC Berkeley. The second author would
like to thank the Department of Mathematics at UC Berkeley for
hospitality during his visit in the Spring of 2022. We also thank the
referee for careful reading of the paper and helpful comments.

\section{Completions}\label{cp}

In this section we introduce various completions of ring $\Yim$
introduced in equation \eqref{YY} which we will need in order to
define the Weyl group action.

\subsection{Lie algebra and Weyl group}    \label{Liealg}

Let $\Glie$ be a finite-dimensional simple Lie algebra
  of rank $n$. We will use the notation and conventions
    of \cite{ka}.
We denote $C = (C_{i,j})_{i,j\in I}$ its Cartan matrix where 
$I=\{1,\ldots, n\}$. Let
$\{\alpha_i\}_{i\in I}$, $\{\alpha_i^\vee\}_{i\in I}$,
$\{\omega_i\}_{i\in I}$, $\{\omega_i^\vee\}_{i\in I}$, and
${\mathfrak{h}}$ be the simple roots, the simple coroots, the
fundamental weights, the fundamental coweights, and a fixed Cartan
subalgebra of $\Glie$, respectively.  We set $Q=\oplus_{i\in
  I}\Z\alpha_i$, $Q^+=\oplus_{i\in I}\Z_{\ge0}\alpha_i$,
$P=\oplus_{i\in I}\Z\omega_i$.  We have the set of roots
$\Delta\subset Q$, which is decomposed as $\Delta = \Delta_+ \sqcup
\Delta_-$ where $\Delta_\pm = \pm (\Delta\cap Q^+)$.  Let
$D=\mathrm{diag}(d_1,\ldots,d_n)$ be the unique diagonal matrix such
that $B=DC$ is symmetric and $d_i$'s are relatively prime positive
integers.  We have the partial ordering on $P$ defined by $\omega\leq
\omega'$ if and only if $\omega'-\omega\in Q^+$. Let $W$ be the Weyl group of
$\mathfrak{g}$, generated by simple reflections $s_i$ ($i\in I$).

Throughout this paper, we fix a non-zero complex number $q$ which is
not a root of unity. We set $q_i=q^{d_i}$. 

\subsection{Laurent polynomials}    \label{Laurent}

Following \cite{Fre}, consider the ring $\Yim$ of Laurent polynomials
in variables in $Y_{i,a}, i\in I,a\in\CC^\times$, over $\Z$ (see
formula \eqref{YY}). Let $\mathcal{M}$ be the multiplicative group of
monomials in $\Yim$. We have the group homomorphism $\omega :
\mathcal{M}\rightarrow P$ which assigns to $m = \prod_{i\in I,
  a\in\CC^\times}Y_{i,a}^{u_{i,a}(m)}$ in $\mathcal{M}$ its
$\g$-weight
$$\omega(m) := \sum_{i\in I, a\in\CC^\times} u_{i,a}(m) \omega_i\in P.$$ 
For example, for $i\in I, a\in\CC^\times$, $\omega(Y_{i,a}) =
\omega_i$. We define $A_{i,a}\in\mathcal{M}$ by formula \eqref{Ai}.
Its weight is $\omega(A_{i,a}) = \alpha_i$.

\subsection{Completions by formal series}    \label{compl}

\begin{defi}    \label{compw}
For $w\in W$, we define the free $\Z$-module
$\wt{\mathcal{Y}}^w$ consisting of the formal series
\begin{equation}    \label{dec}
  \sum_{m\in S} a_m \cdot m, \qquad a_m \in \Z,
\end{equation}
where $S$ is any subset of the set ${\mc M}$ of monomials in $\Yim$
such that for any $m \in S$ we have
\begin{equation}    \label{omm}
w \cdot \omega(m) \; \in \; \bigcup_{j = 1,\ldots,N} \{ \lambda_j -
\sum_{i \in I} n_i \al_i \mid n_i \geq 0 \}
\end{equation}
for some integral weights $\la_j, j=1,\ldots,N$, and for any $\omega
\in P$, we have
\begin{equation}
|\{m\in S \mid \omega(m) = \omega\text{ and } a_m\neq 0\}|
< \infty.
\end{equation}
\end{defi}

The following result is straightforward.

\begin{lem}
$\wt{\mathcal{Y}}^w$ is a complete topological
ring with respect to the natural topology induced by the partial
ordering on the sets of weights of the form \eqref{omm}, and it is a
completion of the ring $\Yim$ with respect to this topology.
\end{lem}

\begin{rem}
For example, if $w = e$, then $\wt{\mathcal{Y}}^e$ contains 
$$\mathcal{Y} \; \underset{\mathbb{Z}[A_{i,a}^{\pm 1}]_{i\in I,
  a\in\mathbb{C}^\times}}\otimes \ZZ((A_{i,a}^{- 1}))_{i\in I,
  a\in\mathbb{C}^\times}.$$
\qed
\end{rem}
	
Similarly, we define a completion $(\mathbb{Z}[y_i^{\pm
    1}]_{i\in I})^w$ of $\mathbb{Z}[y_i^{\pm 1}]_{i\in
  I}$. Namely, for each $\al = \sum_{i \in I} c_i \al_i \in \Delta$ set
\begin{equation}    \label{aal}
a_\al := \prod_{i \in I} a_i^{c_i}.
\end{equation}
Then
$$
(\mathbb{Z}[y_i^{\pm
    1}]_{i\in I})^w := \mathbb{Z}[y_i^{\pm 1}]_{i\in I}
\underset{\Z[a_{w^{-1}(\alpha_i)}^{\pm 1}]_{i \in I}}\otimes
    \Z(\!(a^{-1}_{w^{-1}(\alpha_i)})\!)_{i\in I}.
$$

Let us set the weight of a monomial $\prod_{i\in I}y_i^{u_i}$ 
to be $\sum_{i\in I}u_i \omega_i$. The assignment $\varpi_{w}(Y_{i,a}) =
y_i$ extends to a ring homomorphism
\begin{equation}    \label{varpiw}
\varpi_{w} : \wt{\mathcal{Y}}^w \rightarrow (\mathbb{Z}[y_i^{\pm
    1}]_{i\in I})^w.
\end{equation}

\subsection{Uniqueness of solutions of $q$-difference
  equations}

The ring $\wt{\mathcal{Y}}^w$ admits automorphisms $\tau_a$,
$a\in\mathbb{C}^\times$, defined by $\tau_a(Y_{i,b}) = Y_{i,ab}$.  A
family $(U(a))_{a\in\mathbb{C}^\times}$ of elements in
$\wt{\mathcal{Y}}^w$ is said to be {\em $\mathbb{C}^\times$-equivariant}
if
$$U(ba) = \tau_b(U(a))\text{ for any $a,b\in\mathbb{C}^\times$.}$$ 

\begin{defi}    \label{Gw}
Let $\wt{G}^w$ be the subgroup of the group
$(\wt{\mathcal{Y}}^w)^\times$ of invertible elements of
$\wt{\mathcal{Y}}^w$ consisting of elements of the form $A \cdot S$
where $A \in\mathcal{M}$ and $S$ is of the form
$$S = \pm 1 + \sum_{m\in\mathcal{M}; \; w(\omega(m)) <
    0} a_m m, \qquad a_m
\in \Z.$$ 
We have a group homomorphism $\wt{G}^w \rightarrow P$ sending
$A \cdot S$ to the weight $\omega(A)$ of $A$. It will be called the 
weight of $A\cdot S$.
\end{defi}

Note that any $\mathbb{C}^\times$-equivariant family contained in
$\wt{G}^w$ has a well-defined weight. The following can be viewed as a
uniqueness result for solutions of $q$-difference equations in
$\wt{\mathcal{Y}}^w$.

\begin{lem}\label{unique} Suppose that $\chi\in
  \wt{\mathcal{Y}}^w$ is such that $\chi = \Psi
  \tau_{q^{-r}}(\chi)$
for some $r\in\mathbb{Z}$ and $\Psi\in\wt{G}^w$ with weight in $\Delta$. Then
$\chi=0$.
\end{lem}

\begin{proof} We give the proof in the case of $\wt{\mathcal{Y}}^e$;
  the case of $\wt{\mathcal{Y}}^w$ for a general element $w$ is
  similar. Suppose $\chi\neq 0$ and consider a monomial in $\chi$ of
  maximal weight $\nu$. If the weight $\alpha$ of $\Psi$ is in
  $\Delta_-$, then $\nu$ does not appear among the weights of the
  monomials in the expansion of $\Psi\tau_{q^{-r}}(\chi)$, which is a
  contradiction. And if $\alpha \in \Delta_+$, then we apply the same
  argument to the equation $(\Psi)^{-1} \chi = \tau_{q^{-r}}(\chi)$.
\end{proof}

\subsection{The ring $\Pi$}    \label{completion}

Consider the ring
\begin{equation}\label{fembw}   
\pi := \bigoplus_{w\in
    W}(\mathbb{Z}[y_i^{\pm 1}]_{i\in I})^w.
\end{equation}
and the diagonal embedding $\mathbb{Z}[y_i^{\pm 1}] \hookrightarrow
\pi$. The action of the simple reflection $s_i, i \in I$, on ${\mathbb
  Z}[y_j^{\pm 1}]_{j \in I}$ given by formula \eqref{si}, naturally
extends to $s_i^w : (\mathbb{Z}[y_j^{\pm 1}]_{j\in I})^w\rightarrow
(\mathbb{Z}[y_j^{\pm 1}]_{j\in I})^{ws_i}$.

Hence we obtain the following.

\begin{lem}    \label{Wfd}
The action of the Weyl group $W$ on ${\mathbb Z}[y_i^{\pm 1}]_{i \in
  I}$ generated by the simple reflections $s_i, i \in I$, naturally
extends to $\pi$ defined in equation \eqref{fembw}, with $s_i$ acting
as $(s_i^w)_{w\in W}$.
\end{lem}

We will now define a $q$-analogue of $\pi$; namely, the ring
$$
  \Pi :=  \bigoplus_{w\in W} \wt{\mathcal{Y}}^w.
$$
It is equipped with a diagonal embedding $\mathcal{Y} \hookrightarrow
\Pi$ and automorphisms $\tau_a, a\in\mathbb{C}^\times$. For each $w\in
W$, denote by $E_w$ the projection $\Pi\rightarrow
\wt{\mathcal{Y}}^w$. We also denote by $\varpi$ the
  homomorphism $\Pi \to \pi$ which restricts to $\varpi_w$ given by
  formula \eqref{varpiw} on each summand $\wt{\mc Y}^w$ in $\Pi$.

\begin{defi}\label{exsigor}
For $i\in I$ and $a\in\mathbb{C}^\times$, let
\begin{equation}    \label{Sigmai1}
\Sigma_{i,a}^+ :=  \sum_{k\geq 0} \; \prod_{0 < j \leq k}
  A_{i,aq_i^{-2j+2}}^{-1} = 1 + A_{i,a}^{-1} (1 + A_{i,aq_i^{-2}}^{-1}(1 +
  \ldots )),
\end{equation}
where the $k=0$ term in the summation is defined to be $1$, and
\begin{equation}    \label{Sigmaibis1}
\Sigma_{i,a}^- := - \sum_{k > 0} \; \prod_{0 < j \leq k}
  A_{i,aq_i^{2j}} = - A_{i,aq_i^2} (1 + A_{i,aq_i^4}(1 +
  \ldots )).
\end{equation}
Given $w\in W$, we define an element $\Sigma_{i,a}^w$ of
$\wt{\Yim}^w$ as follows: $\Sigma_{i,a}^w = \Sigma_{i,a}^+$ if
$w(\alpha_i)\in \Delta_+$ and $\Sigma_{i,a}^w = \Sigma_{i,a}^-$ if
$w(\alpha_i)\in \Delta_-$. Finally, set
\begin{equation}    \label{totalSigma}
\Sigma_{i,a} := (\Sigma_{i,a}^w)_{w\in W}\in \Pi.
\end{equation}
\end{defi}

The next lemma follows directly from the definition and Lemma
\ref{unique}.
  
\begin{lem}
For each $w \in W$, $\Sigma_{i,a}^w$ is the unique solution in
$\wt{\mc Y}^w$ of the $q$-difference equation
\begin{equation}\label{firstqe1}
\Sigma_{i,a}^w = 1 + A_{i,a}^{-1} \Sigma_{i,aq_i^{-2}}^w.
\end{equation}
In addition, $\Sigma_{i,a}^w$ is invertible in $\wt{\mc Y}^w$ and
$\Sigma_{i,a}$ is invertible in $\Pi$.
\end{lem}

\section{The Weyl group action}    \label{Weylaction}

In this section we introduce automorphisms $\Theta_i$ generating an
action of the Weyl group on $\Pi$ and state our first main result,
Theorem \ref{weylaction}.

Recall from Section \ref{compl} that each $\wt{\mc Y}^w$ is a complete
topological ring, which is a completion of $\Yim$.  Denote
by $\Yim^w$ the image of $\Yim$ in $\wt{\mc Y}^w$. We will use the
same notation for the corresponding image of $\Yim$ in $\Pi$.

\begin{defi}    \label{defTh}
For each $i\in I$ and $w \in W$, we define the ring homomorphism
$$\Theta^w_i : \mathcal{Y}^w \rightarrow \wt{\mc Y}^{ws_i}$$
by $\Theta^w_i(Y_{j,a}^{\pm 1}) := Y_{j,a}^{\pm 1}$ if $j\neq i$ and
\begin{equation}\label{Thetai1}
  \Theta^w_i(Y_{i,a}) := Y_{i,a}A_{i,aq_i^{-1}}^{-1} \;
  \frac{\Sigma^{ws_i}_{i,aq_i^{-3}}}{\Sigma^{ws_i}_{i,aq_i^{-1}}}.
\end{equation}
\end{defi}

\begin{lem}    \label{conthom}
The homomorphism $\Theta^w_i, i \in I$, extends uniquely to a
continuous homomorphism $\wt{\mc Y}^w \rightarrow \wt{\mc Y}^{ws_i}$.
\end{lem}

\begin{proof}
We need to show that for any formal series $\chi\in
\wt{\mathcal{Y}}^w$, the series obtained by replacing each monomial
$m$ in $\chi$ by $\Theta^w_i(m)$ is well-defined in $\wt{\mc
  Y}^{ws_i}$. We do this in two steps. First, we replace each variable
$Y_{j,a}$ in each monomial $m$ of $\chi$ with $T_i(Y_{j,a}) =
Y_{j,a}A_{i,aq_i^{-1}}^{-\delta_{i,j}}$ (see formula \eqref{Ti}). As
the result, each monomial $m$ of weight $\omega$ becomes a monomial of
weight $s_i(\omega)$, so according to Definition \ref{compw} we obtain
a well-defined series in $\wt{\mathcal{Y}}^{ws_i}$. Next, we multiply
each $T_i(Y_{i,a})$ obtained this way by the additional factor
$\frac{\Sigma_{i,aq_i^{-3}}^{ws_i}}{\Sigma_{i,aq_i^{-1}}^{ws_i}}$. By
Definition \ref{exsigor}, the result is a well-defined element of
$\wt{\mathcal{Y}}^{ws_i}$.
\end{proof}

We will use the same notation $\Theta^w_i$ for the homomorphism
$\wt{\mc Y}^w \rightarrow \wt{\mc Y}^{ws_i}$. Taking the direct sum
over all $w \in W$, we obtain the following.

\begin{lem}
For each $i \in I$, the homomorphisms
$$
\Theta^w_i: \wt{\mc Y}^w \rightarrow \wt{\mc Y}^{ws_i}, \qquad w \in W,
$$
combine into a continuous homomorphism
$$
\Theta_i: \Pi \to \Pi.
$$
\end{lem}

The following is the first main result of this paper. Its proof will
be given in Proposition \ref{invol} and in Section \ref{braidrel}.

\begin{thm}    \label{weylaction}
The homomorphisms $\Theta_i, i\in I$, generate an action of the Weyl
group $W$ on $\Pi$.
\end{thm}

The first step is to show that $\Theta_i^2 = \on{Id}$ for all $i \in
I$.

\begin{prop}\label{invol} Each endomorphism $\Theta_i, i \in I$, of
  $\Pi$ is an involution.
\end{prop}

\begin{proof} We first prove that $\Theta_i^2 = \text{Id}$
  on each $\mathcal{Y}^w \subset \Pi, w \in W$. We have
    $\Theta^w_i(Y_{j,a}) = Y_{j,a}$ for all $j \neq i$. The
    homomorphism property of $\Theta^w_i$ implies that we
    only need to show that viewing $Y_{i,a}$ as an element of
    $\mathcal{Y}^w$, we have
\begin{equation}    \label{invw}
  \Theta_i^{ws_i} \circ \Theta_i^w(Y_{i,a}) = Y_{i,a},
\end{equation}
which, according to formula (\ref{Thetai1}) is equivalent to
\begin{equation}    \label{wsi}
\Theta_i^{ws_i} \left( Y_{i,a}A_{i,aq_i^{-1}}^{-1} \;
\frac{\Sigma^{ws_i}_{i,aq_i^{-3}}}{\Sigma^{ws_i}_{i,aq_i^{-1}}}
\right) = Y_{i,a}.
\end{equation}
Formula (\ref{Thetai1}) (with $w$ replaced by $ws_i$) also implies
that
\begin{equation}    \label{ThA}
    \Theta^{ws_i}_i(A_{i,a}^{-1}) = A_{i,aq_i^{-2}} \;
    \frac{\Sigma^{w}_{i,a}}{\Sigma^{w}_{i,aq_i^{-4}}}.
\end{equation}
Using \eqref{firstqe1} for the element $ws_i \in W$, we obtain that
$\Theta^{ws_i}_i(\Sigma^{ws_i}_{i,a}) \in \wt{\mc Y}^w$ is a solution of the
$q$-difference equation
\begin{equation}\label{tsia}
\Theta^{ws_i}_i(\Sigma^{ws_i}_{i,a}) = 1 + A_{i,aq_i^{-2}} \;
\frac{\Sigma^{w}_{i,a}}{\Sigma^{w}_{i,aq_i^{-4}}}
\Theta^{ws_i}_i(\Sigma^{ws_i}_{i,aq_i^{-2}}).
\end{equation}
Note that
$$
A_{i,aq_i^{-2}} \frac{\Sigma^{w}_{i,a}}{\Sigma^{w}_{i,aq_i^{-4r}}}\in
\wt{G}^w
$$
and has weight $\alpha_i\in P$. Therefore, Lemma
\ref{unique} implies that $\Theta^{ws_i}_i(\Sigma^{ws_i}_{i,a})$
is a unique solution of equation (\ref{tsia}) in $\wt{\mc Y}^w$.

On the other hand, using formula \eqref{firstqe1} now for the element
$w \in W$, we find that $1 - \Sigma^{w}_{i,a}$ is also a solution of
(\ref{tsia}) in $\wt{\mc Y}^w$. Indeed,
$$1 + A_{i,aq_i^{-2}} \;
\frac{\Sigma^w_{i,a}}{\Sigma^w_{i,aq_i^{-4}}}
(1 - \Sigma^w_{i,aq_i^{-2}}) 
= 1 - A_{i,aq_i^{-2}} \;
\frac{\Sigma^w_{i,a}}{\Sigma^w_{i,aq_i^{-4}}}A_{i,aq_i^{-2}}^{-1}\Sigma^w_{i,aq_i^{-4}}
= 1 - \Sigma^w_{i,a}.$$
Applying Lemma \ref{unique}, we derive that the two solutions must
be equal, i.e.
\begin{equation}    \label{oneminus}
\Theta^{ws_i}_i(\Sigma^{ws_i}_{i,a}) = 1 - \Sigma^w_{i,a} = -
A_{i,a}^{-1}\Sigma^w_{i,aq_i^{-2}}.
\end{equation}
Substituting this formula, as well as \eqref{Thetai1} and \eqref{ThA},
into the LHS of \eqref{wsi}, we obtain that it, and hence the LHS of
\eqref{invw}, is equal to
$$Y_{i,a}
A_{i,aq_i^{-1}}^{-1}\frac{\Sigma^w_{i,aq_i^{-3}}}{\Sigma^w_{i,aq_i^{-1}}}
\cdot A_{i,aq_i^{-3}}
\frac{\Sigma^{w}_{i,aq_i^{-1}}}{\Sigma^{w}_{i,aq_i^{-5}}} \cdot
\frac{- A_{i,aq_i^{-3}}^{-1}\Sigma^w_{i,aq_i^{-5}}}{-
  A_{i,aq_i^{-1}}^{-1}\Sigma^w_{i,aq_i^{-3}}} = Y_{i,a}.
$$
This proves that $\Theta_i^2 = \text{Id}$ on each $\mathcal{Y}^w
\subset \Pi, w \in W$. The continuity of $\Theta_i$ (see Lemma
\ref{conthom}) implies that the same is true on the entire $\Pi$. This
completes the proof.
\end{proof}

 The braid group relations between the $\Theta_i$'s will
  be proved in the next section. We close this section with the
  statement about the compatibility of the actions of $W$ on the
  completions of the rings of characters and $q$-characters.

\begin{lem}    \label{Pipi}
The projection $\varpi: \Pi \to \pi$ intertwines
the actions of $W$ on the two rings given by Theorem \ref{weylaction}
and Lemma \ref{Wfd}.
\end{lem}

\begin{proof}
It follows from the definitions that
\begin{equation}\label{schar}\varpi_{ws_i} \circ \Theta_i^w = s_i
  \circ \varpi_w.
\end{equation}
This implies the statement of the lemma.
\end{proof}

\section{Braid group relations}    \label{braidrel}

In this section we complete the proof Theorem \ref{weylaction}. Since
we have already proved in Proposition \ref{invol} that the
endomorphisms $\Theta_i, i \in I$, are involutions of $\Pi$, it
suffices to prove that they satisfy the braid group relations
$(\mathcal{R}_{i,j})$ for $i\neq j$. Using the continuity of
$\Theta_i$ as in the proof of Proposition \ref{invol}, it suffices to
show that the braid relations are satisfied on the elements of
$\mathcal{Y}^w \subset \Pi$ for all $w \in W$.  Since each
  $\Theta_i$ is a ring automorphism of $\Pi$, it is sufficient to show
  that the braid relations are satisfied on the generators $Y_{k,a}$
  of $\mathcal{Y}^w$.

In the above proof of Proposition \ref{invol} we explicitly tracked
the component equations in $\wt{\mathcal{Y}}^w$ corresponding to
different $w \in W$. But in this section, whenever possible, we
combine these components into a single equation in $\Pi$. We can go
back and forth between the two presentations because each braid
relation $(\mathcal{R}_{i,j})$ has the form $\Theta_i \Theta_j \ldots
\Theta_i \Theta_j = \Theta_j \Theta_i \ldots \Theta_j \Theta_i$ or
$\Theta_i \Theta_j \ldots \Theta_i = \Theta_j \Theta_i \ldots
\Theta_j$. These automorphisms of $\Pi$ map $\wt{\mathcal{Y}}^w$ to
$\wt{\mathcal{Y}}^{w'}$, where $w' = ws_is_j \ldots s_is_j = ws_js_i
\ldots s_js_i$ or $w' = ws_is_j \ldots s_i = ws_js_i \ldots
s_j$. Since the map $W \to W$ given by $w \mapsto w'$ is a bijection,
we can always disentangle the components of a combined relation
corresponding to different $w \in W$.

In our computations below, unless stated otherwise, every
element we consider is a sum of its components in $\wt{\mathcal{Y}}^w$
for all $w \in W$. For instance, if we write $Y_{i,a}$, we mean the
element of $\Pi$ which is equal to the sum of the elements $Y_{i,a}$
contained in $\mathcal{Y}^w \subset \wt{\mathcal{Y}}^w$ for all $w \in
W$, and if we write $\Sigma_{i,a}$, we mean the sum of the elements
$\Sigma_{i,a}^w \in \wt{\mathcal{Y}}^w$ for all $w \in W$ (as in
formula \eqref{totalSigma}). Every equation we write should be viewed
as an equation in $\Pi$, i.e. a collection of the component equations
in $\wt{\mathcal{Y}}^w$ for all $w \in W$.

With this understanding, formulas (\ref{Thetai1}),
  (\ref{ThA}) and (\ref{oneminus}) are written as follows:
$$\Theta_i(Y_{i,a}) = Y_{i,a}A_{i,aq_i^{-1}}^{-1} \;
  \frac{\Sigma_{i,aq_i^{-3}}}{\Sigma_{i,aq_i^{-1}}}\text{ , }
 \Theta_i(A_{i,a}^{-1}) = A_{i,aq_i^{-2}} \;
    \frac{\Sigma_{i,a}}{\Sigma_{i,aq_i^{-4}}},$$
$$\Theta_i(\Sigma_{i,a}) =  1  - \Sigma_{i,a} =
    -A_{i,a}^{-1}\Sigma_{i,aq_i^{-2}}.$$

\subsection{Some invariant elements}\label{firstinv}

\begin{defi}    \label{TV}
For $i\in I$, $a\in\mathbb{C}^\times$ and $k\geq 0$, define the
element of $\Yim$, 
\begin{equation}    \label{Tiak}
T_{i,a}^{(k)} := Y_{i,a}Y_{i,aq_i^{-2}}\cdots Y_{i,aq_i^{2 (1-k)}} (1
+ A_{i,aq_i}^{-1}(1 + A_{i,aq_i^{-1}}^{-1} (1 + \cdots +
  A_{i,aq_i^{5-2k}}^{-1}(1 + A_{i,aq_i^{3-2k}}^{-1}))$$
$$= Y_{i,a}\cdots Y_{i,aq_i^{2(1-k)}}\sum_{0\leq \alpha\leq k}
V_{i, aq_i}^{(\alpha)},
\end{equation}
 where
$$  
  V_{i,a}^{(\alpha)} := (A_{i,a}A_{i,aq_i^{-2}}\cdots
  A_{i,aq_i^{2-2\alpha}})^{-1}, \quad \al>0, \qquad V_{i,a}^{(0)} := 1.
  $$
We also define
$$
V_{i,a}^{(\alpha)} =
\left(V_{i,aq_i^{-2\alpha}}^{(-\alpha)}\right)^{-1}, \qquad \alpha
< 0.
$$
\end{defi}

For any $\alpha\in\mathbb{Z}$ we have the relations
\begin{equation}\label{vrel}V_{i,a}^{(\alpha + 1)} =
  V_{i,a}^{(\alpha)}A_{i,aq_i^{-2\alpha}}^{-1} =
  V_{i,aq_i^{-2}}^{(\alpha)} A_{i,a}^{-1}.
\end{equation}

If follows from Definition \ref{exsigor} that
\begin{equation}    \label{split}
\begin{split} 
\Sigma_{i,a}^w = \begin{cases} \sum_{\alpha\geq 0}V_{i,a}^{(\alpha)}  & \text{ if $w(\alpha_i)\in \Delta_+$,}
\\ -\sum_{\alpha < 0}V_{i,a}^{(\alpha)}  & \text{ if $w(\alpha_i)\in
  \Delta_-$}. \end{cases}\end{split}
\end{equation}

\begin{prop}\label{fixed} Each element $T_{i,a}^{(k)}$ given by
  \eqref{Tiak} is fixed by $\Theta_i$.
\end{prop}

\begin{proof} For $k = 1$, the image of $Y_{i,a} (1 + A_{i, aq_i}^{-1})$ under $\Theta_i$ is 
$$\frac{ Y_{i,a}A_{i,aq_i^{-1}}^{-1}\Sigma_{i,aq_i^{-3}} + Y_{i,a} \Sigma_{i,aq_i} 
}{\Sigma_{i,aq_i^{-1}}} = \frac{ Y_{i,a}(\Sigma_{i,aq_i^{-1}} - 1) + Y_{i,a} ( 1 + A_{i,aq_i}^{-1}\Sigma_{i,aq_i^{-1}}) 
}{\Sigma_{i,aq_i^{-1}}} = Y_{i,a}( 1 + A_{i,aq_i}^{-1}).$$
Then we have the recurrence relation for $k\geq 1$: 
$$T_{i,aq_i^2}^{(k+1)} = T_{i,a}^{(k)}T_{i,aq_i^2}^{(1)} - T_{i,aq_i^{-2}}^{(k-1)}(Y_{i,a}Y_{i,aq_i^2}A_{i,aq_i}^{-1}).$$
The result follows since $Y_{i,a}Y_{i,aq_i^2}A_{i,aq_i}^{-1}\in\mathbb{Z}[Y_{j,b}^{\pm 1}]_{j\neq i, b\in\mathbb{C}^\times}$ is fixed by $\Theta_i$.
\end{proof}

\subsection{Reduction to the rank $2$ Lie algebras}    \label{reduc}

 As we explained at the beginning of Section
  \ref{braidrel}, in order to prove that a braid relation
  $(\mathcal{R}_{i,j})$ (with $i\neq j$ in $I$) holds, it is
  sufficient to verify that the relation $(\mathcal{R}_{i,j})$
  applied to $Y_{k,a}$ holds for every $k \in I$ and
  $a\in\CC^\times$. Denote the latter by $({\mc R_{i,j}}(Y_{k,a}))$.

  If $k\notin\{i,j\}$, then $\Theta_i(Y_{k,a}) = \Theta_j(Y_{k,a}) =
  Y_{k,a}$, and so $(\mathcal{R}_{i,j}(Y_{k,a}))$ holds. Hence we only
  need to check that $(\mathcal{R}_{i,j}(Y_{i,a}))$ and
  $(\mathcal{R}_{i,j}(Y_{j,a}))$ hold. Moreover, it follows from the
  definition of $\Theta_i$ and $\Theta_j$ that the result of applying
  them to $Y_{i,a}$ and $Y_{j,a}$ can be expressed in terms of
  iterated solutions of $q$-difference equations involving only
  $A_{i,b}^{\pm 1}$ and $A_{j,b}^{\pm 1}$.

Hence, the action of any successive composition of the automorphisms
$\Theta_i$, $\Theta_j$ on $Y_{i,a}$ and $Y_{j,a}$ for the Lie algebra
$\g$ coincides with the corresponding action for the rank 2 Lie
algebra $\mathfrak{g}_{i,j}$ (associated to the simple roots
$\alpha_i$ and $\alpha_j$) if replace the variables $Y_{i,a}$,
$Y_{j,a}$, $A_{i,b}^{\pm 1}$, and $A_{j,b}^{\pm 1}$ of type
$\mathfrak{g}$ with the corresponding variables of type
$\mathfrak{g}_{i,j}$. Therefore, if the relation $(\mathcal{R}_{i,j})$
holds for $\mathfrak{g}_{i,j}$, then it also holds for
$\mathfrak{g}$.  We thus obtain the following result.

\begin{thm}    \label{weylaction1}
Theorem \ref{weylaction} follows from the braid relation for all
rank 2 semisimple Lie algebras.
\end{thm}

It remains to prove the braid relations for types $A_1\times A_1$,
$A_2$, $B_2$, and $G_2$. In Sections \ref{A1} and \ref{sltb}, we
  prove them combinatorially for types $A_1\times A_1$ and $A_2$,
  respectively. This implies Theorem \ref{weylaction} for all
  simply-laced $\g$. Then in Section \ref{gtb} we prove them uniformly
  for types $A_2, B_2$, and $G_2$ using the $q$-character
  versions of the generalized $TQ$-relations for these Lie algebras
  established in \cite{FH}.

\subsection{Type $A_1\times A_1$}    \label{A1}
Suppose that $C_{i,j} = 0$. Then $({\mc R}_{i,j})$ is
\begin{equation}    \label{A1A1}
\Theta_j\Theta_i = \Theta_i \Theta_j.
\end{equation}
Since $\Theta_j(A_{i,a}) = A_{i,a}$, we obtain that
$\Theta_j(\Sigma_{i,1})$ satisfies the same relation (\ref{firstqe})
as $\Sigma_{i,1}$, and hence $\Theta_j(\Sigma_{i,1}) = \Sigma_{i,1}$
by Lemma \ref{unique}. Therefore
$$(\Theta_j\Theta_i)(Y_{i,1}) = \Theta_i(Y_{i,1}) =
(\Theta_i\Theta_j)(Y_{i,1}),$$ so $(\mathcal{R}_{i,j}(Y_{i,a}))$
holds. Applying the automorphism of $A_1 \times A_1$
exchanging $i$ and $j$, we obtain that $(\mathcal{R}_{i,j}(Y_{j,a}))$
also holds, and so we are done.

To handle the cases with $C_{i,j} < 0$, we need some preliminary results.

\subsection{Image of the $\Sigma_{j,a}$}\label{srel}  Suppose that $C_{i,j} < 0$. We have
\begin{equation*}\begin{split}\Theta_i(A_{j,a}) = A_{j,a} A_{ij,a} \frac{\Sigma_{i,aq_i^{-2-C_{i,j}}}}{\Sigma_{i,aq_i^{-2 + C_{i,j}}}} 
\text{ where } A_{ij,a} = \begin{cases} A_{i,aq_i^{-1}} &\text{ if $C_{i,j} = -1$,} 
\\ A_{i,aq^{-2}}A_{i,a}&\text{ if $C_{i,j} = -2$,}
\\ A_{i,aq}A_{i,aq^{-1}}A_{i,aq^{-3}}&\text{ if $C_{i,j} = -3$.}
\end{cases}\end{split}
\end{equation*}
In particular,
$$\Theta_i(\Sigma_{j,a}) = 1 + A_{j,a}^{-1}A_{ij,a}^{-1} \Theta_i(\Sigma_{j,aq_j^{-2}})\frac{\Sigma_{i,aq_i^{-2+C_{i,j}}}}{\Sigma_{i,aq_i^{-2 - C_{i,j}}}}.$$
The following formulas will be useful:
\begin{equation*}\begin{split}\Theta_i(\Sigma_{j,a}) =  \frac{\Sigma_{ij,a}}{\Sigma_{i,a}^{(j)}} \text{ where }
\Sigma_{i,a}^{(j)} =  \begin{cases} \Sigma_{i,aq_i^{-2-C_{i,j}}} &\text{ if $C_{j,i} = -1$,} 
\\ \Sigma_{i,aq^{-2}}\Sigma_{i,aq^{-4}}&\text{ if $C_{j,i} = -2$,}
\\ \Sigma_{i,aq^{-3}}\Sigma_{i,aq^{-5}}\Sigma_{i,aq^{-7}}&\text{ if $C_{j,i} = -3$,}
\end{cases}\end{split}\end{equation*}
and $\Sigma_{ij,a}$ is the unique solution of the $q$-difference equation
\begin{equation}\label{qdiffs}\Sigma_{ij,a} = \Sigma_{i,a}^{(j)} +
  A_{j,a}^{-1} A_{ij,a}^{-1}   \Sigma_{ij,aq_j^{-2}}.\end{equation}

\subsection{Simply-laced types}\label{sltb}
 In this subsection, we prove the braid relations for type
  $A_2$. By the argument of Section \ref{reduc}, this implies Theorem
  \ref{weylaction} for all simply-laced $\g$.

Suppose that $C_{i,j} = C_{j,i} = -1$ and $d_i = d_j =
1$.  Then $({\mc R}_{i,j})$ is
\begin{equation}    \label{A2rel}
\Theta_i\Theta_j\Theta_i = \Theta_j\Theta_i\Theta_j.
\end{equation}
Applying the automorphism of $A_2$
exchanging $i$ and $j$ to the relation
$(\mathcal{R}_{i,j}(Y_{i,a}))$, we obtain
$(\mathcal{R}_{i,j}(Y_{j,a}))$. Hence it is sufficient to prove 
$(\mathcal{R}_{i,j}(Y_{i,a}))$, which is the following

\begin{prop}    \label{A2}
  $$
  (\Theta_i\Theta_j\Theta_i)(Y_{i,a}) =
  (\Theta_j\Theta_i\Theta_j)(Y_{i,a}).
  $$
\end{prop}

\begin{proof}
Since $\Theta_i(Y_{j,a}) = Y_{j,a}$, we need to prove that 
$$(\Theta_i\Theta_j\Theta_i)(Y_{i,a}) = (\Theta_j\Theta_i)(Y_{i,a}).$$

Equation (\ref{qdiffs}), with $i$ and $j$ exchanged, reads in this
case:
\begin{equation}    \label{qdiffs1}
\Sigma_{ji,a} = \Sigma_{j,aq^{-1}} + A_{i,a}^{-1} A_{j,aq^{-1}}^{-1}
\Sigma_{ji,aq^{-2}}.
\end{equation}
We will prove that 
\begin{equation}    \label{Sigmaji}
  \Theta_i(\Sigma_{ji,a}) = \Sigma_{ji,a},
\end{equation}
so that we have the following commutative diagram:

$$\begin{xymatrix}{ Y_{i,a} \ar@{|->}[rr]^{\Theta_i}\ar@{|->}[d]_{\Theta_j} &&  Y_{i,a}A_{i,aq^{-1}}^{-1}\frac{\Sigma_{i,aq^{-3}}}{\Sigma_{i,aq^{-1}}}\ar@{|->}[rr]^{\Theta_j} &&Y_{j,aq^{-3}}^{-1}\frac{\Sigma_{ji,aq^{-3}}}{\Sigma_{ji,aq^{-1}}}\ar@{|->}[d]^{\Theta_i}
    \\  Y_{i,a}\ar@{|->}[rr]_{\Theta_i} && Y_{i,a}A_{i,aq^{-1}}^{-1}\frac{\Sigma_{i,aq^{-3}}}{\Sigma_{i,aq^{-1}}} \ar@{|->}[rr]_{\Theta_j} &&Y_{j,aq^{-3}}^{-1}\frac{\Sigma_{ji,aq^{-3}}}{\Sigma_{ji,aq^{-1}}}}
\end{xymatrix}$$

First, Lemma \ref{unique} implies that
\begin{equation}\label{relij}\Sigma_{i,a}\Sigma_{j,aq} = \Sigma_{ij,aq} + A_{j,aq}^{-1}\Sigma_{ji,a}\end{equation}
because both sides of (\ref{relij}) satisfy the same $q$-difference
equation (for $U(a)$):
$$U(a) = A_{j,aq}^{-1}A_{i,a}^{-1} U(aq^{-2}) + (\Sigma_{i,a} + \Sigma_{j,aq} - 1).$$
Now, since $\Theta_i(\Sigma_{j,a}) =
\Sigma_{ij,a}\Sigma_{i,aq^{-1}}^{-1}$, applying $\Theta_i$ to
(\ref{relij}), we obtain
$$-A_{i,a}^{-1}\Sigma_{i,aq^{-2}}\Sigma_{ij,aq}\Sigma_{i,a}^{-1} = -\Sigma_{j,aq}A_{i,a}^{-1}\Sigma_{i,aq^{-2}} + A_{j,aq}^{-1}A_{i,a}^{-1}\Theta_i(\Sigma_{ji,a})\Sigma_{i,aq^{-2}}\Sigma_{i,a}^{-1}.$$
and so
$$
\Theta_i(\Sigma_{ji,a}) = A_{j,aq}(-\Sigma_{ij,aq} +
\Sigma_{j,aq}\Sigma_{i,a}) = \Sigma_{ji,a}.
$$
\end{proof}

\subsection{General types}\label{gtb}

The proof in the previous section is based on the $q$-difference
equations \eqref{qdiffs} and \eqref{qdiffs1}. We expect that there is
a similar proof for other types as well, but so far we have not been
able to find the relevant equations for types $B_2$ and $G_2$.

Hence we present here a different but uniform proof for the rank $2$
simple types. So, suppose that $\mathfrak{g}$ is of type $A_2$, $B_2$
or $G_2$.  The proof is based on the generalized Baxter $TQ$-relations
proved in \cite{FH}. To explain this, let us consider the extension
$$
\Pi' := \mathcal{Y}'\otimes_\mathcal{Y} \Pi
$$
of $\Pi$ where
$$
\mathcal{Y}' :=
\mathbb{Z}[\Psib_{{k},a}^{\pm 1}, [\omega]]_{k\in
    I,a\in\mathbb{C}^\times,\omega\in P} \supset \mathcal{Y}
$$
and we set
  \begin{equation}\label{ysub}
    Y_{k,a} =
  [\omega_{k}]\Psib_{{k},aq_{k}^{-1}}\Psib_{{k},aq_{k}}^{-1}.
\end{equation}
Thus, $\mathcal{Y}'$ is the ring of Laurent polynomials in the
algebraically independent variables  $\Psib_{{k},a}$ over the
ring $\mathbb{Z}[P]$.

The weight of a Laurent monomial $[\omega]\Psib_{k_1,a_1}^{p_1}\cdots
\Psib_{k_N,a_N}^{p_N}$ in $\mathcal{Y}'$ is defined to be $\omega$.

Like $\Pi$, the ring $\Pi'$ has components $\wt{\mathcal{Y}}'{}^w,
w\in W$, and the corresponding projections.

Next, we extend the operators $\Theta_k$ on $\Pi$ to the
operators $\Theta_k'$ on  $\Pi'$ given by the following
formulas: for $\omega\in P$ and $a\in\mathbb{C}^\times$
\begin{equation}    \label{Thetaprime}
\Theta_{k}'([\omega]) := [s_{k}(\omega)], \qquad
\Theta_{k}'(\Psib_{{m},a}) := 
\left\{ 
\begin{array}{lc}
\Psib_{{m},a} & \text{ if $m \neq k$}, \\[2mm]
\widetilde{\Psib}_{{k},aq_{k}^{-2}}\Sigma_{{k},aq_{k}^{-2}} & \mbox{ if $m = k$}.
\end{array}
\right.
\end{equation}
where $\wt{\Psib}_{{k},aq_{k}^{-2}}$, introduced in \cite{FH2}, is given
by the formula
\begin{multline}\label{psfh2}\wt{\Psib}_{{k},aq_{k}^{-2}} :=
    \Psib_{{k},aq_{k}^{-2}}^{-1} \cdot \\ \prod_{m\in I,C_{k,m} =
      -1}\Psib_{m,aq_k^{-1}} \cdot \prod_{m\in I, C_{k,m} =
  -2}\Psib_{m,aq^{-2}}\Psib_{m,a} \cdot \prod_{m\in I, C_{k,m} =
  -3}\Psib_{m,aq^{-3}}\Psib_{m,aq^{-1}}\Psib_{m,aq}.\end{multline}

These operators are well-defined and according to formula
(\ref{ysub}), they are compatible with the operators $\Theta_{k}$.

We set $I = \{ i,j \}$ with $d_i \geq d_j$. We are going to prove
that the operators $\Theta'_i$ and $\Theta'_j$ satisfy the braid
relations ${\mc R}'_{i,j,\ell} = {\mc R}'_{i,j,r}$, where 
$${\mc R}'_{i,j,\ell}
:= \Theta_j'\Theta_i'\Theta_j' \quad \text{
  and } \quad
{\mc R}'_{i,j,r} :=
\Theta_i'\Theta_j'\Theta_i'\text{ in type $A_2$,}$$
$${\mc R}'_{i,j,\ell}
:= \Theta_j'\Theta_i'\Theta_j'\Theta_i' \quad \text{
  and } \quad
{\mc R}'_{i,j,r} :=
\Theta_i'\Theta_j'\Theta_i'\Theta_j'\text{ in type $B_2$,}$$
$${\mc R}'_{i,j,\ell}
:= \Theta_j'\Theta_i'\Theta_j'\Theta_i'\Theta_j'\Theta_i' \quad \text{
  and } \quad
{\mc R}'_{i,j,r} :=
\Theta_i'\Theta_j'\Theta_i'\Theta_j'\Theta_i'\Theta_j'\text{ in type $G_2$.}$$
This will imply the sought-after braid relations.

The above relation is clearly satisfied on the elements $[\omega]$,
$\omega\in P$. Let us now check this relation on the elements
$\Psib_{i,a}$ and $\Psib_{j,a}$. Explicitly, formula
\eqref{Thetaprime} specializes to
$$
\Theta_i'(\Psib_{i,a}) =
\Psib_{i,aq^{-2d_i}}^{-1}\Psib_{j,aq^{-d_i}}\Sigma_{i,aq^{-2d_i}}, \qquad \Theta_i'(\Psib_{j,a}) = \Psib_{j,a},\qquad \Theta_j'(\Psib_{i,a}) =
\Psib_{i,a},$$ 
$$\Theta_j'(\Psib_{j,a}) = \Psib_{j,aq^{-2}}^{-1}\Sigma_{j,aq^{-2}}\times \begin{cases}
\Psib_{i,aq^{-1}}\text{ in type $A_2$,}
\\\Psib_{i,aq^{-2}}\Psib_{i,a}\text{ in type $B_2$,}
\\\Psib_{i,aq^{-3}}\Psib_{i,aq^{-1}}\Psib_{i,aq}\text{ in type $G_2$.}\end{cases}
$$

Introduce the following notation for $a\in\mathbb{C}^\times$ : 
\begin{equation}    \label{philr}
\phi^\ell_{i,a} := {\mc R}'_{i,j,\ell}(\Psib_{i,a}) \text{ , }
\phi^\ell_{j,a} := {\mc
  R}'_{i,j,\ell}(\Psib_{j,a}) \text{ , } \phi^r_{i,a} := {\mc
  R}'_{i,j,r}(\Psib_{i,a}) \text{ , } \phi^r_{j,a} := {\mc
  R}'_{i,j,r}(\Psib_{j,a}).
\end{equation}
It remains to prove that
\begin{equation}    \label{remains}
\phi^\ell_{i,a} = \phi^r_{i,a} \qquad \text{and} \qquad 
\phi^\ell_{j,a} = \phi^r_{j,a}.
\end{equation}

 We will consider the projections onto
  $\wt{\mathcal{Y}}'{}^e$ (the proof for the other components
  $\wt{\mathcal{Y}}'{}^e, w \in W$, is similar).

Consider the sets of words 
$$W_i := \{(i), (ji), (iji)\},\qquad W_j := \{(j), (ij), (jij)\}\text{ in type $A_2$,}$$ 
$$W_j := \{(j), (ij), (jij), (ijij)\},\qquad W_i := \{(i), (ji), (iji), (jiji)\}\text{ in type $B_2$,}$$ 
$$W_j := \{(j), (ij), (jij), (ijij), (jijij),(ijijij)\},$$
$$W_i := \{(i), (ji), (iji), (jiji), (ijiji),(jijiji)\}\text{ in type $G_2$.}$$
Then
$\Theta_w'$, written as the product of the $\Theta'_i$ and
  $\Theta'_j$ according to the sequence $w$, is well-defined for any such
word $w$. Note that we have ${\mc R}'_{i,j,\ell} =
\Theta'_{w_j}$ and ${\mc R}'_{i,j,r} = \Theta'_{w_i}$, where $w_j$ (resp. $w_i$) is the longest word in $W_j$ (resp. $W_i$).

We introduce intermediate elements $\Sigma_{w,a}\in \Pi$ for each $w\in W_i\cup W_j$ and $a\in\mathbb{C}^*$. 
For $w$ of length $1$ or $2$, $\Sigma_{w,a}$ has been defined above.

In type $A_2$, we define $\Sigma_{iji,a}$ and $\Sigma_{jij,a}$ by the formulas 
$$\Theta_i(\Sigma_{ji,a}) = \Sigma_{iji,a},\qquad 
\Theta_j(\Sigma_{ij,a}) = \Sigma_{jij,a}.$$

In type $B_2$, we define $\Sigma_{iji,a}$, $\Sigma_{jiji,a}$, $\Sigma_{jij,a}$, and $\Sigma_{ijij,a}$ by the formulas 
$$\Theta_i(\Sigma_{ji,a}) = \frac{\Sigma_{iji,a}}{\Sigma_{i,aq^{-2}}},\qquad \Theta_j(\Sigma_{iji,a}) = \Sigma_{jiji,a},\qquad 
\Theta_j(\Sigma_{ij,a}) =
\frac{\Sigma_{jij,a}}{\Sigma_{j,aq^{-4}}},\qquad
\Theta_i(\Sigma_{jij,a}) = \Sigma_{ijij,a}.$$

In type $G_2$, we define $\Sigma_{iji,a}$, $\Sigma_{jiji,a}$,
$\Sigma_{ijiji,a}$, and $\Sigma_{jijiji,a}$ by the formulas 
$$\Theta_i(\Sigma_{ji,a}) =
\frac{\Sigma_{iji,a}}{\Sigma_{i,aq^{-4}}\Sigma_{i,aq^{-2}}} \qquad
\Theta_j(\Sigma_{iji,a}) = \frac{\Sigma_{jiji,a}}{\Sigma_{j,aq^{-3}}},$$
$$\Theta_i(\Sigma_{jiji,a}) = \frac{\Sigma_{ijiji,a}}{\Sigma_{i,aq^{-6}}}, \qquad \Theta_j(\Sigma_{ijiji,a}) = \Sigma_{jijiji,a},$$
and we define $\Sigma_{jij,a}$, $\Sigma_{ijij,a}$, $\Sigma_{jijij,a}$, and $\Sigma_{ijijij,a}$ by the formulas 
$$\Theta_j(\Sigma_{ij,a}) = \frac{\Sigma_{jij,a}}{\Sigma_{j,aq^{-4}}\Sigma_{j,aq^{-6}}}\text{ , }
\Theta_i(\Sigma_{jij,a}) = \frac{\Sigma_{ijij,a}}{\Sigma_{i,aq^{-7}}\Sigma_{i,aq^{-9}}\Sigma_{i,aq^{-11}}},$$
$$\Theta_j(\Sigma_{ijij,a}) = \frac{\Sigma_{jijij,a}}{\Sigma_{j,aq^{-10}}}\text{ , }\Theta_i(\Sigma_{jijiji,a}) 
= \Sigma_{ijijiji,a}.$$

It follows from the above definitions that for each $w\in W_i$
(resp. $w\in W_j$), $\Theta_w'(\Psib_{i,a})\Sigma_{w,aq^{-2d_i}}^{-1}$
(resp. $\Theta_w'(\Psib_{j,a})\Sigma_{w,aq^{-2}}^{-1}$) is a Laurent
monomial in the $\Psib_{k,b}^{\pm 1}; k=i,j$.

In type $A_2$, the family $\Sigma_{ji,a}$ satisfies the $q^2$-difference equation
$$\Sigma_{ji,a} = \Sigma_{j,aq^{-1}} +
A_{i,a}^{-1}A_{j,aq^{-1}}^{-1}\Sigma_{ji,aq^{-2}}.$$
The family $\Sigma_{ij,a}$ satisfies the $q^2$-difference equation
$$\Sigma_{ij,a} = \Sigma_{i,aq^{-1}} +
A_{j,a}^{-1}A_{i,aq^{-1}}^{-1} \Sigma_{ij,aq^{-2}}.$$

Therefore, we have the following $q^2$-difference equations
$$\Sigma_{iji,a} \Sigma_{i,aq^{-2}} =
  \Sigma_{ij,aq^{-1}} +
  A_{j,aq^{-1}}^{-1}\Sigma_{i,a}
  \Sigma_{iji,aq^{-2}},$$
	$$\Sigma_{jij,a} \Sigma_{j,aq^{-2}}=
\Sigma_{ji,aq^{-1}} +
A_{i,aq^{-1}}^{-1}\Sigma_{j,a}\Sigma_{jij,aq^{-2}}.$$

In type $B_2$, the family $\Sigma_{ji,a}$ satisfies the $q^4$-difference equation
\begin{equation}    \label{Sigij2}
\Sigma_{ji,a} = \Sigma_{j,a} +
A_{i,a}^{-1}A_{j,aq^{-2}}^{-1}A_{j,a}^{-1}\Sigma_{ji,aq^{-4}}.
\end{equation}
The family $\Sigma_{ij,a}$ satisfies the $q^2$-difference equation
\begin{equation}    \label{Sigij1}
\Sigma_{ij,a} = \Sigma_{i,aq^{-2}}\Sigma_{i,aq^{-4}} +
A_{j,a}^{-1}A_{i,aq^{-2}}^{-1} \Sigma_{ij,aq^{-2}}.
\end{equation}

Therefore, we have the following $q^4$-difference equations
\begin{equation}\label{qde5}\Sigma_{iji,a} \Sigma_{i,aq^{-4}} =
  \Sigma_{ij,a} +
  A_{i,aq^{-2}}^{-1}A_{j,aq^{-2}}^{-1}A_{j,a}^{-1}\Sigma_{i,a}
  \Sigma_{iji,aq^{-4}},\end{equation}
$$\Sigma_{jiji,a}\Sigma_{ji,aq^{-4}} = \Sigma_{jij,a} + A_{i,aq^{-2}}^{-1}\Sigma_{ji,a}\Sigma_{jiji,aq^{-4}},$$
and the following $q^2$-difference equations
\begin{equation}    \label{Sigiji}
\Sigma_{jij,a} \Sigma_{j,aq^{-2}}=
\Sigma_{ji,aq^{-2}}\Sigma_{ji,aq^{-4}} +
A_{i,aq^{-2}}^{-1}A_{j,aq^{-4}}^{-1}\Sigma_{j,a}\Sigma_{jij,aq^{-2}},
\end{equation}
$$\Sigma_{ijij,a} \Sigma_{ij,aq^{-2}}=
\Sigma_{iji,aq^{-2}}\Sigma_{iji,aq^{-4}} + A_{j,aq^{-4}}^{-1}\Sigma_{ij,a}\Sigma_{ijij,aq^{-2}}.$$

In type $G_2$, the family $\Sigma_{ji,a}$ satisfies the $q^6$-difference equation
\begin{equation}\label{gqdji}\Sigma_{ji,a} = \Sigma_{j,aq} + A_{i,a}^{-1}A_{j,aq}^{-1}A_{j,aq^{-1}}^{-1}A_{j,aq^{-3}}^{-1}\Sigma_{ji,aq^{-6}},\end{equation}
The family $\Sigma_{ij,a}$ satisfies the $q^2$-difference equation
\begin{equation}\label{gqdij}\Sigma_{ij,a} = \Sigma_{i,aq^{-3}}\Sigma_{i,aq^{-5}}\Sigma_{i,aq^{-7}} + A_{j,a}^{-1}A_{i,aq^{-3}}^{-1}\Sigma_{ij,aq^{-2}},\end{equation}

Therefore, we have the following $q^6$-difference equations
\begin{equation}\label{gqdiji}\Sigma_{iji,a}\Sigma_{i,aq^{-6}} = \Sigma_{ij,aq} 
+ A_{i,aq^{-4}}^{-1}A_{i,aq^{-2}}^{-1}A_{j,aq^{-3}}^{-1}A_{j,aq^{-1}}^{-1}A_{j,aq}^{-1}\Sigma_{iji,aq^{-6}}\Sigma_{i,a},\end{equation}
\begin{equation}\label{gqdjiji}\Sigma_{jiji,a}\Sigma_{ji,aq^{-6}} = \Sigma_{jij,aq} 
+ A_{i,aq^{-4}}^{-1}A_{i,aq^{-2}}^{-1}A_{j,aq^{-7}}^{-1}A_{j,aq^{-5}}^{-1}A_{j,aq^{-3}}^{-1}\Sigma_{jiji,aq^{-6}}\Sigma_{ji,a},\end{equation}
\begin{equation}\label{gqdijiji}\Sigma_{ijiji,a}\Sigma_{iji,aq^{-6}} = \Sigma_{ijij,aq} 
+ A_{i,aq^{-6}}^{-1}A_{j,aq^{-7}}^{-1}A_{j,aq^{-5}}^{-1}A_{j,aq^{-3}}^{-1}\Sigma_{ijiji,aq^{-6}}\Sigma_{iji,a},\end{equation}
\begin{equation}\label{gqdjijiji}\Sigma_{jijiji,a}\Sigma_{jiji,aq^{-6}} = \Sigma_{jijij,aq} 
+ A_{i,aq^{-6}}^{-1}\Sigma_{jijiji,aq^{-6}}\Sigma_{jiji,a},\end{equation}
and the following $q^2$-difference equations
\begin{equation}\label{gqdjij}\Sigma_{jij,a}\Sigma_{j,aq^{-2}} = \Sigma_{ji,aq^{-3}}\Sigma_{ji,aq^{-5}} \Sigma_{ji,aq^{-7}} 
+ A_{j,aq^{-4}}^{-1}A_{i,aq^{-3}}^{-1}A_{j,aq^{-6}}^{-1}\Sigma_{jij,aq^{-2}}\Sigma_{j,a},\end{equation}
\begin{equation}\label{gqdijij}\Sigma_{ijij,a}\Sigma_{ij,aq^{-2}} = \Sigma_{iji,aq^{-3}}\Sigma_{iji,aq^{-5}} \Sigma_{iji,aq^{-7}} 
+ A_{j,aq^{-4}}^{-1}A_{i,aq^{-7}}^{-1}A_{j,aq^{-6}}^{-1}\Sigma_{ijij,aq^{-2}}\Sigma_{ij,a},\end{equation}
\begin{equation}\label{gqdjijij}\Sigma_{jijij,a}\Sigma_{jij,aq^{-2}} = \Sigma_{jiji,aq^{-3}}\Sigma_{jiji,aq^{-5}} \Sigma_{jiji,aq^{-7}} 
+ A_{j,aq^{-10}}^{-1}A_{i,aq^{-7}}^{-1}\Sigma_{jijij,aq^{-2}}\Sigma_{jij,a},\end{equation}
\begin{equation}\label{gqdijijij}\Sigma_{ijijij,a}\Sigma_{ijij,aq^{-2}} = \Sigma_{ijiji,aq^{-3}}\Sigma_{ijiji,aq^{-5}} \Sigma_{ijiji,aq^{-7}} 
+ A_{j,aq^{-10}}^{-1}\Sigma_{ijijij,aq^{-2}}\Sigma_{ijij,a}.\end{equation}

In all of these equations, the first factor in the second term
on the right-hand side is a monomial in $A_{k,c}^{-1}$ whose weight is a
negative root. This implies, by induction on the length of $w$, that
  $$
  \Sigma_{w,a} \in \Z[[A_{k,c}^{-1}]]^{(1)}_{k \in I, c \in \C^\times}
  $$
where $\Z[[A_{k,c}^{-1}]]^{(1)}_{k \in I, c \in \C^\times}$ consists of
all elements of $\Z[[A_{k,c}^{-1}]]_{k \in I, c \in \C^\times}$ whose
highest weight term is equal to $1$. Recalling formulas
\eqref{philr}, we obtain the following result.

\begin{lem}    \label{Laur}
  We have, for $k=i,j$,
  $$
  \phi^\ell_{k,a} \in \phi^\ell_{k,a,0} \cdot \Z[[A_{k,c}^{-1}]]^{(1)}_{k \in I, c
    \in \C^\times}, \qquad \phi^r_{k,a} \in \phi^r_{k,a,0} \cdot
  \Z[[A_{k,c}^{-1}]]^{(1)}_{k \in I, c \in \C^\times},
  $$
  where $\phi^\ell_{k,a,0}$ and $\phi^r_{k,a,0}$ are monomials
  in $\Psib_{i,b}^{\pm 1}$ and $\Psib_{j,b}^{\pm 1}$.
\end{lem}

We will prove that the highest weight terms $\phi^\ell_{k,a,0}$ and
$\phi^r_{k,a,0}$ are equal to each other for $k=i,j$. We will then use
this result to prove that $\phi^\ell_{k,a} = \phi^r_{k,a}$ for $k =
i,j$.

Consider the following elements of $\Yim \subset \Pi$: in type $A_2$
$$T_{i,a} = Y_{i,a} + Y_{i,aq^2}^{-1}Y_{j,aq} + Y_{j,aq^3}^{-1},\qquad T_{j,a} = Y_{j,a} + Y_{j,aq^2}^{-1}Y_{i,aq} + Y_{i,aq^3}^{-1},$$
in type $B_2$
$$T_{i,a} = Y_{i,a} + Y_{i,aq^4}^{-1}Y_{j,aq}Y_{j,aq^3} + Y_{j,aq}Y_{j,aq^5}^{-1} + Y_{j,aq^3}^{-1}Y_{j,aq^5}^{-1}Y_{i,aq^2} + Y_{i,aq^6}^{-1},$$
$$T_{j,a} = Y_{j,a} + Y_{j,aq^2}^{-1}Y_{i,aq} + Y_{i,aq^5}^{-1} Y_{j,aq^4} + Y_{j,aq^6}^{-1},$$
and in type $G_2$
\begin{equation}\label{bi}
T_{i,a} = Y_{i,a} + Y_{i,aq^6}^{-1}Y_{j,aq}Y_{j,aq^3}Y_{j,aq^5} + Y_{j,aq}Y_{j,aq^3}Y_{j,aq^7}^{-1} + Y_{i,aq^4}Y_{j,aq}Y_{j,aq^5}^{-1}Y_{j,aq^7}^{-1}\end{equation}
$$+Y_{i,aq^2}Y_{i,aq^4}Y_{j,aq^3}^{-1}Y_{j,aq^5}^{-1}Y_{j,aq^7}^{-1} + Y_{i,aq^{10}}^{-1}Y_{j,aq}Y_{j,aq^9} + Y_{i,aq^2}Y_{i,aq^{10}}^{-1}Y_{j,aq^3}^{-1}Y_{j,aq^9}$$ 
$$+ Y_{i,aq^8}^{-1}Y_{i,aq^4} + Y_{i,aq^8}^{-1}Y_{i,aq^{10}}^{-1}Y_{j,aq^5}Y_{j,aq^7}Y_{j,aq^9} + Y_{i,aq^8}^{-1}Y_{j,aq^5}Y_{j,aq^7}Y_{j,aq^{11}}^{-1} $$
$$+ Y_{j,aq^5}Y_{j,aq^9}^{-1}Y_{j,aq^{11}}^{-1} + Y_{i,aq^6}Y_{j,aq^7}^{-1}Y_{j,aq^9}^{-1}Y_{j,aq^{11}}^{-1}+ Y_{i,aq^{12}}^{-1},$$
\begin{equation}\label{bj}T_{j,a} = Y_{j,a} + Y_{j,aq^2}^{-1}Y_{i,aq}
  + Y_{i,aq^7}^{-1}Y_{j,aq^4}Y_{j,aq^6} +
  Y_{j,aq^4}Y_{j,aq^8}^{-1} \end{equation}
$$+ Y_{j,aq^6}^{-1}Y_{j,aq^8}^{-1}Y_{i,aq^5} 
+ Y_{i,aq^{11}}^{-1}Y_{j,aq^{10}} + Y_{j,aq^{12}}^{-1}.$$

\begin{lem}    \label{invar}
The elements $T_{i,a}$ and $T_{j,a}$ are invariant under the action of
$\Theta_i$ and $\Theta_j$.
\end{lem}

\begin{proof}
The elements $T_{i,a}$ and $T_{j,a}$ are the $q$-characters of the
$i$th and $j$th fundamental representations of $U_q(\wh{\mathfrak{g}})$
(recall that we denote the two nodes of the Dynkin diagram of $\mathfrak{g}$ by
$i$ and $j$). This follows from the algorithm of \cite{Fre2} (see
\cite[Section 8.4]{H}).
We prove their invariance under $\Theta_i$ and $\Theta_j$ following
the argument of Theorem \ref{invariants} below. Namely, by
construction (see \cite[Corollary 5.7]{Fre2}), each of them can be
written as a polynomial in
$$Y_{i,a}(1 + A_{i,aq_i}^{-1}) \quad \on{and} \quad Y_{j,a}^{\pm 1}$$
for various $a\in\mathbb{C}^\times$. We have
$\Theta_i(Y_{j,a}^{\pm 1}) = Y_{j,a}^{\pm 1}$ by definition,
and since $Y_{i,a}(1 + A_{i,aq_i}^{-1}) = T^{(1)}_{i,a}$ (see formula
\eqref{Tiak}), Proposition \ref{fixed} implies that
$$\Theta_i(Y_{i,a}(1 + A_{i,aq_i}^{-1})) = Y_{i,a}(1 +
A_{i,aq_i}^{-1}).$$
The proof of invariance under $\Theta_j$ is similar, and this
completes the proof.
\end{proof}

Next, recall the generalized Baxter $TQ$-relations proved in
\cite{FH}. These relations appear in the Grothendieck ring of the
category ${\mc O}$ introduced in
\cite{HJ}, which is an extension of the category of finite-dimensional
representations of $U_q(\ghat)$. Our proof below is based
  on a combinatorial study of the images of the $TQ$-relations under
  the $q$-character homomorphism; we will not use any information
  about representations from the category ${\mc O}$. They are as follows:
$$T_{k,a} = \mathcal{S}_k(([\omega_{i}],[\omega_{j}],
\Psib_{i,b},\Psib_{j,b})_{b\in\mathbb{C}^\times}), \qquad k=i,j,$$
where the right hand side is obtained by substituting the right hand
side of formula \eqref{ysub} for $Y_{i,a}$ and $Y_{j,a}$ in the 
defining formulas for $T_{i,a}$ and $T_{j,a}$.

Let us apply ${\mc R}'_{i,j,\ell}$ and ${\mc R}'_{i,j,r}$ to these
equations. By Lemma \ref{invar}, $T_{i,a}$ and $T_{j,a}$ are invariant
under ${\mc R}'_{i,j,\ell}$ and ${\mc R}'_{i,j,r}$, and we have ${\mc
  R}'_{i,j,\ell}([\omega_k]) = {\mc R}'_{i,j,r}([\omega_k]) =
[-\omega_{\overline{k}}]$ for $k = i,j$, where $\overline{k} = k$ in types $B_2$, $G_2$ and $\overline{k}\neq k$ in type $A_2$. Hence we obtain the following system of two
equations in $\Pi$ on $\phi^{\ell /r}_{i,b}$ and $\phi^{\ell /r}_{j,b}$:
\begin{equation}    \label{Tia}
\mathcal{S}_i(([\omega_i],[\omega_j],
\Psib_{i,b},\Psib_{j,b})_{b\in\mathbb{C}^\times}) =
\mathcal{S}_i(([-\omega_{\overline{i}}],[-\omega_{\overline{j}}],
\phi^{\ell /r}_{i,b},\phi^{\ell /r}_{j,b})_{b\in\mathbb{C}^\times}),
\end{equation}
\begin{equation}    \label{Tja}
\mathcal{S}_j(([\omega_i],[\omega_j],
\Psib_{i,b},\Psib_{j,b})_{b\in\mathbb{C}^\times}) =
\mathcal{S}_j(([-\omega_{\overline{i}}],[-\omega_{\overline{j}}],
\phi^{\ell /r}_{i,b},\phi^{\ell /r}_{j,b})_{b\in\mathbb{C}^\times}),
\end{equation}
where notation $\ell /r$ means that the formula is valid if we put
$\ell$ everywhere or $r$ everywhere. We are going to derive from this
system the desired equalities \eqref{remains}.

Recall that we are considering the projections onto
$\wt{\mathcal{Y}}'{}^e$. Let $D = 3$ in type $A_2$ (resp $D = 6$ in
type $B_2$ and $D = 12$ in type $G_2$).

It is clear that only the highest weight
terms of
$$
{\mc R}'_{i,j,\ell /r}(Y_{\overline{i},aq^{D}}^{-1}) =
[\omega_i]\phi^{\ell /r}_{\overline{i},aq^{D  +
    d_{\overline{i}}}}(\phi^{\ell /r}_{\overline{i},aq^{D - d_{\overline{i}}}})^{-1} 
$$
contribute to the
highest weight term on the right hand side of \eqref{Tia}. But the
highest weight term on the left hand side of \eqref{Tia} is $Y_{i,a} =
[\omega_i]\Psib_{i,aq^{-d_i}}\Psib_{i,aq^{d_i}}^{-1}$. Therefore, the
highest weight terms of $\phi^\ell_{\overline{i},a}$ and $\phi^r_{\overline{i},a}$ are both
equal to $\Psib_{i,aq^{-D}}^{-1}$.

Using equation \eqref{Tja} in the same way, we find that the highest
weight terms of $\phi^\ell_{\overline{j},a}$ and $\phi^r_{\overline{j},a}$ are both equal to
$\Psib_{j,aq^{-D}}^{-1}$.

Combining this with Lemma \ref{Laur}, we obtain the following result.

\begin{lem}    \label{Laur1}
We have
\begin{equation}    \label{assum}
\phi^{\ell /r}_{i,b} 
\in \Psib_{\overline{i},bq^{-D}}^{-1} \cdot \mathbb{Z}[[A_{k,c}^{-1}]]^{(1)}_{k\in I,
  c\in\mathbb{C}^\times} \qquad
\phi^{\ell /r}_{j,b} \in
  \Psib_{\overline{j},bq^{-D}}^{-1} \cdot \mathbb{Z}[[A_{k,c}^{-1}]]^{(1)}_{k\in I,
    c\in\mathbb{C}^\times}.
\end{equation}
\end{lem}

We can now prove the equalities \eqref{remains}.

\begin{lem}
There is a unique solution $(\phi^{\ell /r}_{i,b},\phi^{\ell
  /r}_{j,b})$ of the system
\eqref{Tia}, \eqref{Tja} in $\wt{\mathcal{Y}}'{}^e$ of
  the form \eqref{assum}.
\end{lem}

\begin{proof} We have already shown the existence of such
  solutions in Lemma \ref{Laur1}. Now we prove uniqueness.
To simplify our notation, in this proof we will write $\phi_{i,b}$
(resp. $\phi_{j,b}$) for $\phi^{\ell
  /r}_{i,b}$ (resp. $\phi^{\ell /r}_{j,b}$).

By induction on descending weights, we can show that all lower weight
terms in $\phi_{i,a}$ and $\phi_{i,a}$ are uniquely determined by
equations \eqref{Tia} and \eqref{Tja}. 
Indeed, let us expand
$$\phi_{i,b} =  \sum_{m\geq 0} \phi_{i,b,m} \qquad
\phi_{j,b} = \sum_{m\geq 0} \phi_{j,b,m},$$ where $\phi_{i,b,0} =
\Psib_{\overline{i},bq^{-D}}^{-1}, \phi_{j,b,0} =
\Psib_{\overline{j},bq^{-D}}^{-1}$, and for $m>0$, each expression
$\phi_{i,b,m}$ (resp. $\phi_{j,b,m}$) is a finite linear combination
of terms equal to the product of
$\Psib_{\overline{i},bq^{-D}}^{-1}$ (resp. $\Psib_{\overline{j},bq^{-D}}^{-1}$)
and a monomial of the form
$A_{k_1,c_1}^{-1} \ldots A_{k_m,c_m}^{-1}$; note that the weight of such a term
is equal to the sum of the corresponding $m$ negative simple
roots: $-\sum_{a=1}^m \al_{k_a}$. We will prove uniqueness of
$\phi_{i,b,m}$, $\phi_{j,b,m}$ by induction on $m\geq 0$.

The terms $\phi_{i,b,0}$ and $\phi_{j,b,0}$ corresponding
  to $m=0$ are fixed by formula \eqref{assum} (which is
  our assumption in the present lemma); namely, they
  are equal to $\Psib_{\overline{i},bq^{-D}}^{-1}$ and $\Psib_{\overline{j},bq^{-D}}^{-1}$,
  respectively. Suppose now
that we have found the terms $\phi_{i,b,m'}$ and $\phi_{j,b,m'}$ with
$m'<m$. Consider the part of equation \eqref{Tia} comprising all terms
whose weights are of the form $\omega_i - \gamma$, where $\gamma$ is a
sum of $m$ simple roots. It is clear that only $\phi_{i,b,m'},
\phi_{j,b,m'}$ with $m' \leq m$ can contribute to this part of
equation \eqref{Tia}.

By computing the weights of these terms, we find that the
  contribution of $\phi_{i,b,m}, \phi_{j,b,m}$ to this part of
  equation \eqref{Tia} comes exclusively from ${\mc R}'_{i,j,\ell
    /r}(Y_{\overline{i},aq^{D}}^{-1}) =
  [\omega_i]\phi_{\overline{i},aq^{D +
      d_{\overline{i}}}}\phi_{\overline{i},aq^{D  -
        d_{\overline{i}}}}^{-1}$, and hence this
  contribution is equal to $[\omega_i](\phi_{\overline{i},aq^{D + d_{\overline{i},m}}} -
  \phi_{\overline{i},aq^{D - d_{\overline{i}}},m})$.

Therefore, we obtain from equation \eqref{Tia} that
$[\omega_i](\phi_{\overline{i},aq^{D + d_{\overline{i}},m} -
\phi_{\overline{i},aq^{D - d_{\overline{i}}},m}})$ can be expressed in terms of
$\phi_{i,b,m'}, \phi_{j,b,m'}$ with $m' < m$. This implies that
$\phi_{\overline{i},b,m}$ is uniquely determined by $\phi_{i,b,m'},
\phi_{j,b,m'}$ with $m'< m$. Indeed, if there were two
  solutions $\phi^{(1)}_{\overline{i},b,m}$ and $\phi^{(2)}_{\overline{i},b,m}$, then their
  difference $\wt\phi_{\overline{i},b,m}$ would satisfy the corresponding homogeneous
equation $[\omega_i](\wt\phi_{\overline{i},aq^{D + d_{\overline{i}}},m} - \wt\phi_{\overline{i},aq^{D - d_{\overline{i}}},m}) =
0$. This equation implies that $\wt\phi_{\overline{i},b,m}$
does not depend on $b$ and hence is in $\Z[[\omega]]_{\omega \in P}$.
But according to formula \eqref{assum},
$\wt\phi_{\overline{i},b,m}$ is a linear combination of monomials
of the form $\Psib_{\overline{i},bq^{-D}}^{-1} A_{k_1,c_1}^{-1}\cdots
A_{k_m,c_m}^{-1}, m>0$. Therefore, we obtain that $\wt\phi_{\overline{i},b,m}=0$, and
so $\phi_{\overline{i},b,m}$ is uniquely determined by $\phi_{i,b,m'},
\phi_{j,b,m'}$ with $m'< m$.

Applying the same argument to the corresponding part of
  equation \eqref{Tja}, we find that $\phi_{\overline{j},b,m}$ is also uniquely
  determined by $\phi_{i,b,m'}, \phi_{j,b,m'}$ with $m'< m$. This
  completes the inductive step and hence the proof.
\end{proof}

This lemma implies the equalities \eqref{remains}. This completes the
proof of the braid relation.

\section{Subring of $W$-invariants, $q$-characters, and screening
  operators}\label{screenings}

In this section we prove that the subring of the invariants in $\Yim$
of the action of the Weyl group on $\Pi$ coincides with the ring of
$q$-characters inside $\Yim$, which is isomorphic to the Grothendieck
ring of the category of finite-dimensional representations of
$\U_q(\ghat)$. For that, we relate the operators $\Theta_i$ to the
screening operators constructed in \cite{Fre} and use the results of
\cite{Fre2} that the subring of invariants of the screening
operators in $\Yim$ coincides with the ring of $q$-characters.

\subsection{Quantum affine algebras and their finite-dimensional
  representations}    \label{rep}

In this section we collect some definitions and results on quantum
affine algebras and their representations. We refer the reader to
\cite{CP} for a canonical introduction, and to \cite{CH,L} for more
recent surveys on this topic.

Let $\wh{\Glie}$ be the Kac--Moody Lie algebra of untwisted affine
type associated to $\Glie$.  The quantum affine algebra
$U_q(\wh{\Glie})$ is a Hopf algebra over $\C$ defined in
  terms of the Drinfeld--Jimbo generators $e_i,\ f_i,\ k_i^{\pm1}$
  ($0\le i\le n$); see, e.g., \cite[Section 2.1]{FH}.
 We are interested in its level 0 quotient, which we
  denote by $\U_q(\wh{\Glie})$. It has a presentation \cite{Dri2, bec,
    da} in terms of the Drinfeld generators
\begin{align*}
  x_{i,r}^{\pm}\ (i\in I, r\in\Z), \quad \phi_{i,\pm m}^\pm\ (i\in I,
  m\geq 0), \quad k_i^{\pm 1}\ (i\in I).
\end{align*}
We will use the generating series $(i\in I)$: 
$$\phi_i^\pm(z) = \sum_{m\geq 0}\phi_{i,\pm m}^\pm z^{\pm m} =
k_i^{\pm 1}\text{exp}\left(\pm (q_i - q_i^{-1})\sum_{m > 0} h_{i,\pm
    m} z^{\pm m} \right).$$
The algebra $\U_q(\wh{\Glie})$ has a $\ZZ$-grading defined by
$\on{deg}(x_{i,m}^\pm) = \on{deg}(\phi_{i,m}^\pm) = m$ for $i\in I$,
  $m\in\ZZ$, and $\on{deg}(k_i^{\pm 1}) = 0$ for $i\in I$.
For $a\in\CC^\times$, there is a corresponding automorphism 
$\tau_a$ of $\U_q(\wh{\Glie})$ so that for an element $g$ of degree
$m\in\ZZ$ satisfies $\tau_a(g) = a^m g$.

For $i\in I$, the action of $k_i$ on any object of the category $\mathcal{F}$ of finite-dimensional representations of $\U_q(\wh{\Glie})$ is diagonalizable with eigenvalues in $\pm q^{d_i\ZZ}$.
Without loss of generality, we can assume that $\mathcal{F}$ is the category of {\it type 1} finite-dimensional representations (see \cite{CP}), i.e. we assume that for any object of $\mathcal{F}$, the eigenvalues of $k_i$ are in $q^{d_i \ZZ}$ for $i\in I$. 
The simple objects of $\mathcal{F}$ have been classified by
Chari--Pressley (see \cite{CP}). The simple objects are parametrized
by $n$-tuples of polynomials $(P_i(u))_{i\in I}$ satisfying $P_i(0) =
1$ (they are called {\it Drinfeld polynomials}).

For $\omega\in P$, the {\it weight space} $V_\omega$ of an object $V$ in $\mathcal{F}$ is the set of {\it weight vectors} of weight $\omega$, i.e. of vectors $v\in V$ satisfying $k_i v = q^{\left(d_i \omega(\alpha_i^\vee)\right)} v$ for any $i\in I$. Thus, we have the weight space decomposition 
$$V = \bigoplus_{\omega \in P} V_\omega.$$ 
Let us define the $\ell$-weight decomposition which is its refinement.

Since we consider representations of the level-zero quotient
$\U_q(\wh{\Glie})$,
the Drinfeld-Cartan elements $h_{i,r}$, $i\in I$,
$r\in\mathbb{Z}\setminus \{0\}$, commute on any object $V$ of
$\mathcal{F}$. Since the elements $h_{i,r}$ also commute with the
elements $k_i$, $i\in
I$, every object in $\mathcal{F}$ can be decomposed as a direct sum of
generalized eigenspaces of the $h_{i,r}$ and $k_i$.

More precisely, it follows from
the Frenkel--Reshetikhin theory of $q$-characters \cite{Fre} that the eigenvalues of the $h_{i,r}$ and $k_i$ can be {\it encoded} by {\it monomials} $m\in\mathcal{M}$.  Given $m\in\mathcal{M}$ and an object $V$ in $\mathcal{F}$, let $V_m$ be the corresponding generalized eigenspace of $V$ (also called {\it $\ell$-weight spaces}); thus, 
$$V = \bigoplus_{m\in\mathcal{M}} V_m.$$  
If $v\in V_m$, then $v$ is a weight vector of weight $\omega(m)\in P$.

The {\it $q$-character homomorphism} is an injective ring homomorphism
$$\chi_q : \text{Rep}(\U_q(\Glie)) \rightarrow \Yim \text{ , } \qquad
\chi_q(V) = \sum_{m\in \mathcal{M}} \text{dim}(V_m) m.$$
If $V_m\neq \{0\}$ we say that $m$ is an {\it $\ell$-weight of $V$}.

A monomial $m\in\mathcal{M}$ is said to be {\it dominant} if
$u_{i,a}(m)\geq 0$ for any $i\in I, a\in\CC^\times$. Given a simple
object $V$ in $\mathcal{F}$, let $M(V)$ be the {\it highest weight
  monomial} of $\chi_q(V)$, i.e. such that $\omega(M(V))$ is maximal
for the partial ordering on $P$. It is known that $M(V)$ is dominant
and characterizes the isomorphism class of $V$ (it is, in fact,
equivalent to the data of the Drinfeld polynomials). Conversely, to a
dominant monomial $M$ is associated a unique (up to an isomorphism)
simple object $L(M)$ in $\mathcal{F}$. For $i\in I$ and
$a\in\CC^\times$, we set $V_i(a) :=
L\left(Y_{i,a}\right)$; this is the $i$th fundamental representation
of $\U_q(\ghat)$.

For example, the $q$-character of the fundamental representation
$V_{1,a} = L(Y_{1,a})$ of $\U_q(\widehat{sl}_2)$ is
$$\chi_q(L(Y_{1,a})) = Y_{1,a} + Y_{1,aq^2}^{-1}.$$
It was proved in  \cite{Fre, Fre2} that for a simple module $L(m)$ we have
$$\chi_q(L(m))\in m\ZZ[A_{i,a}^{-1}]_{i\in I, a\in\CC^\times}.$$

\subsection{Screening operators and the ring of $W$-invariants}

In this subsection we link the operators $\Theta_i$
  generating the action of the Weyl group to the screening operators
  $S_i$ introduced in \cite[Section 7]{Fre}. Recall that in Section
  \ref{Weylaction} we denoted by $\Yim^w$ the image of $\Yim$ in the
  completion $\wt{\mc Y}^w$ and in the direct sum $\Pi$ of these
  completions for all $w \in W$. Then in Definition \ref{defTh} we
  define the ring homomorphisms $\Theta^w_i: \mathcal{Y}^w \rightarrow
  \wt{\mc Y}^{ws_i}$ using formula \eqref{Thetai1}. Now
    we define the ring homomorphisms
$$
{\mb \Theta}_i: \Yim \to \wt{\mc Y}^e, \qquad i \in I,
$$
by ${\mb \Theta}_i(Y_{j,a}^{\pm 1}) = Y_{j,a}^{\pm 1}$ if $j\neq
i$ and
\begin{equation}\label{Thetai2}
  {\mb \Theta}_i(Y_{i,a}) = Y_{i,a}A_{i,aq_i^{-1}}^{-1} \;
  \frac{\Sigma^e_{i,aq_i^{-3}}}{\Sigma^e_{i,aq_i^{-1}}}.
\end{equation}
 The connection between ${\mb \Theta}_i$ and $\Theta_i$ is
  as follows: ${\mb \Theta}_i$ is the composition of the isomorphism
  $\Yim \simeq \Yim^{s_i}$ and the homomorphism $\Theta_i^{s_i}:
  \Yim^{s_i} \to \wt{\mc Y}^e$ from Definition \ref{defTh}.

Now let $h$ be a formal variable and set
$$\wt{\Sigma}_{i,a} := h \Sigma^e_{i,a}\text{ for $i\in I$ and
  $a\in\mathbb{C}^\times$}.$$
We then have
$${\mb \Theta}_i(Y_{i,a}) = Y_{i,a}A_{i,aq_i^{-1}}^{-1}
\frac{\Sigma^e_{i,aq_i^{-3}}}{\Sigma^e_{i,aq_i^{-1}}}
= Y_{i,a} - h\frac{Y_{i,a}}{\wt{\Sigma}_{i,aq_i^{-1}}},$$
$${\mb \Theta}_i(Y_{i,a}^{-1}) = Y_{i,a}^{-1}A_{i,aq_i^{-1}}
\frac{\Sigma^e_{i,aq_i^{-1}}}{\Sigma^e_{i,aq_i^{-3}}} =
Y_{i,a}^{-1} + h
\frac{Y_{i,a}^{-1}A_{i,aq_i^{-1}}}{\wt{\Sigma}_{i,aq_i^{-3}}}.$$
Let us also set
$$
\mathcal{Y}_{i,h} :=  \mathcal{Y}[\wt{\Sigma}_{i,a}^{-1},
  h]_{a\in\mathbb{C}^\times}.
$$
Then the above formulas give rise to a ring homomorphism
$${\mb \Theta}_{i,h} : \mathcal{Y}\rightarrow \mathcal{Y}_{i,h}.$$ 
Consider the limit $ h \rightarrow 0$. Note that the relation
$\wt{\Sigma}_{i,a} = h + A_{i,a}^{-1}\wt{\Sigma}_{i,aq_i^{-2}}$
then becomes
$\wt{\Sigma}_{i,a} = A_{i,a}^{-1}\wt{\Sigma}_{i,aq_i^{-2}}$. Therefore, 
\begin{equation}    \label{hdep}
  {\mb \Theta}_{i,h} = \text{Id} + h S_i + h^2(\ldots),
\end{equation}
where $S_i : \mathcal{Y}\rightarrow \mathcal{Y}_i$ is a derivation, and $\mathcal{Y}_i$ is the $\mathcal{Y}$-module generated by the 
$\wt{\Sigma}_{i,a}^{-1}$ with the relation $\wt{\Sigma}_{i,a}^{-1} = A_{i,a}\wt{\Sigma}_{i,aq_i^{-2}}^{-1}$ and 
$$S_i(Y_{j,a}) = -\delta_{i,j} Y_{i,a}\wt{\Sigma}_{i,aq_i^{-1}}^{-1}\text{ , }S_i(Y_{j,a}^{-1}) = \delta_{i,j} Y_{i,a}^{-1}\wt{\Sigma}_{i,aq_i^{-1}}^{-1}.$$
By denoting $S_{i,a} = - \wt{\Sigma}_{i,aq_i^{-1}}^{-1}$, we obtain
the following relation in the limit $h \to 0$:
$$
S_{i,aq_i^2} = A_{i,aq_i}S_{i,a},
$$
and so $S_i$ gets identified with the screening operator defined in
\cite{Fre}.
Thus, the screening operator $S_i$ appears in the limit of a
one-parameter deformation of the operator ${\mb
    \Theta}_i$ defined using formula \eqref{Thetai2}.

We will now use this to prove that the ring of $q$-characters
(equivalently, the Grothendieck ring of the category of
finite-dimensional representations of $\U_q(\ghat)$) is equal to the
subring of $W$-invariants in $\Yim$ embedded diagonally
into $\Pi$.

\begin{thm}    \label{invariants}
  The image of the $q$-character homomorphism $\chi_q$ in $\Yim$ is
  equal to the subring of invariants of the diagonal
    subspace $\Yim \subset \Pi$ under the action of $\Theta_i, i \in
  I$, i.e.
  $$\on{Im}(\chi_q) = \bigcap_{i\in I} \mathcal{Y}^{\Theta_i}.$$
  Equivalently, the Grothendieck ring of the category of finite-dimensional
representations of $\U_q(\ghat)$ is isomorphic to the subring of
invariants in $\Yim$ of the action of the Weyl group on $\Pi$.
\end{thm}

\begin{proof} First, we prove that the elements in $\on{Im}(\chi_q)$
  are invariant under $\Theta_i$ for each $i \in
  I$. 
Indeed, let $i\in I$ and
$L\in\on{Im}(\chi_q)$. According to \cite[Corollary 5.7]{Fre2}, $L$ is
a polynomial in
$$Y_{i,a}(1 + A_{i,aq_i}^{-1}) \quad \on{and} \quad Y_{j,a}^{\pm 1}, j\neq i$$
for various $a\in\mathbb{C}^\times$. We have
$\Theta_i(Y_{j,a}^{\pm 1}) = Y_{j,a}^{\pm 1}$ by definition,
and since $Y_{i,a}(1 + A_{i,aq_i}^{-1}) = T^{(1)}_{i,a}$ (see formula
\eqref{Tiak}), Proposition \ref{fixed} implies that
$$\Theta_i(Y_{i,a}(1 + A_{i,aq_i}^{-1})) = Y_{i,a}(1 +
A_{i,aq_i}^{-1}).$$
Thus, we find that $\on{Im}(\chi_q)$ is contained in $\bigcap_{i\in I}
\mathcal{Y}^{\Theta_i}$.

Conversely, let $L$ be an element in the intersection $\bigcap_{i\in
  I} \mathcal{Y}^{\Theta_i}$. Then ${\mb \Theta}_i(L) = L$, and
therefore ${\mb \Theta}_{i,h}(L) = L$, for all $i \in I$. But then
formula \eqref{hdep} implies that $S_i(L) = 0$ for all $i\in I$. Hence
$L \in \bigcap_{i\in I} \on{Ker}_{\Yim} S_i$. By \cite[Theorem 5.1]{Fre2},
$$
\on{Im}(\chi_q) = \bigcap_{i\in I} \on{Ker}_{\Yim} S_i.
$$
Hence $L\in\on{Im}(\chi_q)$. This completes the proof.
\end{proof}

\begin{rem}
  The subring of $W$-invariants in $\Pi$ is larger than
  the ring of $q$-characters. This is already so for
  $\g=sl_2$. Indeed, according to formula \eqref{oneminus} the element
  $\Sigma_{1,a} (1 - \Sigma_{1,a})$ of $\Pi$ is invariant under
  $\Theta_1$. But it is not in $\Yim$.\qed
\end{rem}

\section{Relation with other symmetries}\label{othsym}

In this section we study the relation of the action of the Weyl group
$W$ defined above with other known symmetries. In particular, we
introduce a $q$-analogue of a natural ring of rational functions
carrying an action of $W$.

\subsection{Ring of rational fractions}

 Recall that in Section \ref{compl} we defined completions
  $$
  (\mathbb{Z}[y_j^{\pm 1}]_{j\in I})^w, \quad w \in W
  $$
  of $\mathbb{Z}[y_j^{\pm 1}]_{j\in I}$. In Lemma
  \ref{Wfd} we extended the natural action of the Weyl group $W$ on
  ${\mathbb Z}[y_i^{\pm 1}]_{i \in I}$ (generated by the simple
  reflections $s_i, i \in I$, given by formula \eqref{si}) to the
  direct sum $\pi$ of these completions defined in equation
  \eqref{fembw} (into which ${\mathbb Z}[y_i^{\pm 1}]_{i \in I}$ is
  embedded diagonally).

Now we introduce a ring ${\mc R}$ which lies in-between
${\mathbb Z}[y_i^{\pm 1}]_{i \in I}$ and $\pi$, and is preserved by
the action of $W$. Namely, we set
\begin{equation}    \label{R}
{\mc R} := \Z[y_i^{\pm
    1},(1-a_\al^{-1})^{-1}]_{i\in I,\al\in \Delta_+}.
\end{equation}
We have a natural embedding
\begin{equation}    \label{embw}
{\mc R} \hookrightarrow (\mathbb{Z}[y_i^{\pm
    1}]_{i\in I})^w, \qquad w \in W
\end{equation}
obtained by replacing $(1-a_\al^{-1})^{-1}$ with
$$
\sum_{n \geq 0} a_\al^{-n} \quad \on{if} \al \in w^{-1}(\Delta_+)
$$
and with
$$
- \sum_{n > 0} a_\al^n \quad \on{if} \al \in w^{-1}(\Delta_-).
$$
 Thus, we have natural embeddings
\begin{equation}    \label{embedd}
{\mathbb Z}[y_i^{\pm 1}]_{i \in I} \subset {\mc R} \subset \pi,
\end{equation}
where the last embedding is diagonal.

\begin{example} For $\Glie$ of type $A_2$ with $I = \{i,j\}$, we have
  $\frac{1}{(1 - a_{\alpha_j}^{-1})(1-a_{\alpha_i+ \alpha_j}^{-1})}$
  in ${\mc R}$. This is $\varpi(\Sigma_{ji,a})$. Its expansion in
  $(\mathbb{Z}[y_i^{\pm 1}]_{i\in I})^w$ is
\begin{equation}\begin{split}
\begin{cases} \sum_{0\leq \beta\leq \alpha } a_j^{-\alpha}a_{i}^{-\beta}&\text{ if $w(\alpha_j)\in \Delta_+$ and $w(\alpha_i + \alpha_j)\in \Delta_+$,}
\\-\sum_{0\leq \beta, \alpha <\beta} a_{j}^{-\alpha}a_i^{-\beta}&\text{ if $w(\alpha_j)\in \Delta_-$ and $w(\alpha_i + \alpha_j)\in \Delta_+$,}
\\-\sum_{\beta< 0,\beta\leq \alpha } a_{j}^{-\alpha}a_{i}^{-\beta}&\text{ if $w(\alpha_j)\in \Delta_+$ and $w(\alpha_i + \alpha_j)\in \Delta_-$,}
\\ \sum_{\alpha < \beta < 0} a_{j}^{-\alpha}a_{i}^{-\beta}&\text{ if $w(\alpha_j)\in \Delta_-$ and $w(\alpha_i + \alpha_j)\in \Delta_-$.}
\end{cases}\end{split}
\end{equation}
These correspond to the different cases in (\ref{exp1}).
\end{example}

Let us extend the action of simple reflections $s_i, i \in I$, on
${\mathbb Z}[y_i^{\pm 1}]_{i \in I}$ given by formula \eqref{si}, to
${\mc R}$ as follows. For each $\al \in \Delta_+$, the element $1-a_\al$ is 
invertible in ${\mc R}$ with the inverse
$$
(1-a_\al)^{-1} = - a_\al^{-1} (1-a_\al^{-1})^{-1}.
$$
Since $\Delta$
is stable under $W$, the action of $W$ on ${\mathbb Z}[y_i^{\pm 1}]_{i
  \in I}$ extends to ${\mc R}$. Moreover, we obtain that the
embeddings \eqref{embedd} commute with the action of $W$.

 The ring ${\mc R}$ is meaningful from the point of view
  of representation theory of the Lie algebra $\g$ because the
  characters of $\g$-modules from the category ${\mc O}$ and its
  twists ${\mc O}_w, w \in W$, discussed in the introduction, may be
  viewed as elements of ${\mc R}$ expanded in a particular completion
  $({\mathbb Z}[y_i^{\pm 1}]_{i \in I})^w$ of ${\mathbb Z}[y_i^{\pm
      1}]_{i \in I}$ (that's because the character of every module
  from the category ${\mc O}$ is known to be a linear combination of
  the characters of Verma modules). Finite-dimensional $\g$-modules
  correspond to elements of the ring ${\mathbb Z}[y_i^{\pm 1}]_{i \in
    I}$ itself.

We will now define a $q$-analogue $\ol{\mc Y}\subset \Pi$ of 
${\mc R} \subset \pi$ and a $q$-analogue of the
embeddings \eqref{embw}. It will then follow from Lemma \ref{Pipi}
that the homomorphism $\ol{\mc Y} \to {\mc R}$ obtained by restriction
of the homomorphism $\varpi: \Pi \to \pi$ intertwines the actions of
$W$ on $\ol{\mc Y}$ and ${\mc R}$.

\subsection{Solutions of $q$-difference equations}

 Recall the group $\wt{G}^w$ from Definition \ref{Gw}.

\begin{lem}\label{unisol} Fix an integer $r\in\mathbb{Z}$ and let $G(a)$,
  $H(a)$, $F(a)$ be families of $\mathbb{C}^\times$-equivariant
  elements in $\wt{G}^w$ of respective weights $\alpha\in \Delta$,
  $0$ and $0$. Then there is a unique $\mathbb{C}^\times$-equivariant
  family $U(a)$ of elements in $\wt{\mathcal{Y}}^w$ such that
\begin{equation}\label{funcrel}
    F(a) U(a) = H(a) + G(a)
  U(aq^{-r})\text{ for any $a\in\mathbb{C}^\times$.}
\end{equation}
Moreover, we have $U(a)\in \wt{G}^w$ for any
$a\in\mathbb{C}^\times$.
\end{lem}

\begin{proof} We give the proof for
    $\wt{\mathcal{Y}}^e$; the proof for an arbitrary
    $\wt{\mathcal{Y}}^w$ is analogous. Since $H(a)/F(a)$ has weight
  $0$ and $G(a)/F(a)$ has weight $\alpha$, dividing by $F(a)$, we can
  assume without loss of generality that $F(a) = 1$. Next, we can
  assume without loss of generality that $H(a) = 1$. Indeed, the above
  equation with $F(a)=1$ is equivalent to the equation
$$U'(a) = 1 + G'(a) U'(aq^{-r})$$
where
$$U'(a) = U(a)/H(a), \qquad G'(a) = G(a) H(aq^{-r})/H(a),$$
and $G'(a)\in \wt{G}^e$ has the same weight as $G(a)$. So let us set
$F(a) = H(a) = 1$.

First, let us prove the existence and uniqueness of the solution $U(a)
\in \wt{G}$ of the equation
\begin{equation}\label{funcrel1}
U(a) = 1 + G(a) U(aq^{-r})
\end{equation}
provided that $G(a)$ is a $\C^\times$-equivariant family of weight
$\al \in \Delta$.

Let us factorize
$$U(a) = u(a) \tilde{U}(a)\text{ and }G(a) = g(a) \tilde{G}(a)$$ where
$u(a), g(a)$ are the highest weight monomials of $U(a)$, $G(a)$, respectively. 
Let $\gamma = \omega(u(a))$ be
  the weight of $u(a)$ and let us look at the terms of highest weight
  in (\ref{funcrel1}). It is clear from (\ref{funcrel1}) that the
  highest weight could only be $0$, $\gamma$ or $\gamma + \alpha$.

If $\alpha\in \Delta_+$, then the highest weight is $\gamma + \alpha >
\gamma$. Then we must have $\gamma + \alpha = 0$ and $0 = 1 + g(a)
u(aq^{-r})$, so we find that $\gamma = - \alpha$ and $u(a) = -
g(aq^{r})^{-1}$. Equation (\ref{funcrel1}) can therefore be written as
$$\tilde{G}(a)^{-1}(g(aq^{r})^{-1} \tilde{U}(a) +  1 - \tilde{G}(a) )
= \tilde{U}(aq^{-r}) - 1.$$

Denote by $(\tilde{U}(a))_{\lambda}$ the term of weight $\lambda$ in
the expansion of $\tilde{U}(a)$. The above equation gives rise to a
system of recurrence relations for $\lambda < 0$:
$$(\tilde{U}(aq^{-r}))_{\lambda} = \sum_{\mu \leq 0}
(\tilde{G}(a)^{-1})_{\mu} \left( g(aq^{r})^{-1} \; (\tilde{U}(a))_{\lambda
  - \mu + \alpha} + (1 - \tilde{G}(a) )_{\lambda - \mu} \right).$$
All the terms $(\tilde{U}(a))_\nu$ appearing on the right hand side
have weights $\nu = \lambda - \mu + \alpha > \lambda$. Hence these
equations determine uniquely all the terms
$(\tilde{U}(a))_{\lambda}, \lambda<0$, of $\tilde{U}(a)$.

If $\alpha\in \Delta_-$, then the highest weight must be $\gamma = 0$
and we must have $u(a) = 1$. Equation (\ref{funcrel1}) becomes
$$\tilde{U}(a) - 1  =  g(a) \tilde{G}(a) \tilde{U}(aq^{-r}).$$
Hence for $\lambda < 0$, we have recurrence relations
$$(\tilde{U}(a))_{\lambda} = \sum_{\mu \leq 0} g(a) (\tilde{G}(a))_\mu
(\tilde{U}(aq^{-r}))_{\lambda - \mu -\alpha}.$$ This again determines
uniquely all the terms $(\tilde{U}(a))_{\lambda}, \lambda<0$, of
$\tilde{U}(a)$.

It remains to prove that any solution of (\ref{funcrel1}) is
necessarily of this form. Suppose that there are two solutions $U(a)$
and $V(a)$ in $\wt{\mathcal{Y}}^w$. Then the difference $D(a) = U(a) -
V(a)$ satisfies $D(a) = G(a) D(aq^{-r})$. 
By Lemma \ref{unique}, we have $D(a) = 0$.
 This completes the proof.
\end{proof}

\subsection{The algebra $\ol{\mc Y}$}\label{comp2}
We define by induction a sequence of $\mathcal{Y}$-subalgebras of $\Pi$: 
$$\mathcal{Y}^0\hookrightarrow \mathcal{Y}^1\hookrightarrow
\mathcal{Y}^2\hookrightarrow \cdots$$
together with subgroups of invertible elements $G^m \subset (\mathcal{Y}^m)^\times$ such that

(1) the algebra $\mathcal{Y}^m$ is invariant under the automorphisms
$\tau_a$ ($a\in\mathbb{C}^\times$),

(2) for any $w\in W$, we have $E_w(G^m)\subset \wt{G}^w$.

\noindent We will also define
\begin{equation}    \label{adm}
(G^m)^{adm} := \{g\in G^m \mid \forall w\in W, \varpi_w(E_w(g)) = a_\al \;
  \on{for} \; \on{some} \; \al \in \Delta \},
\end{equation}
where $a_\al$ is given by formula \eqref{aal}.

First, we define $\mathcal{Y}^0$ (resp. $G^0$) as the image of
$\mathcal{Y}$ (resp. $\mathcal{M}$) in $\Pi$ under the diagonal
embedding.
Suppose now that $G^m\subset \mathcal{Y}^m$ have been defined.
Let $g\in (G^m)^{adm}$, $a\in\mathbb{C}^\times$,
$r\in\mathbb{Z}\setminus\{0\}$, and $w\in W$.
 By Lemma \ref{unisol}, there is a unique invertible
  $f^w_{g,r}(a)\in \wt{G}^w$ such that
 \begin{equation}\label{rell1}
  f^w_{g,r}(a) = 1 + E_w(\tau_a(g))
  f^w_{g,r}(aq^{-r}).\end{equation}
Hence $f_{g,r}(a) = (f^w_{g,r}(a))_{w\in W}$ is the unique invertible
solution in $\Pi$ of
\begin{equation}\label{rell1-1}
  f_{g,r}(a) = 1 + \tau_a(g)
  f_{g,r}(aq^{-r}).
\end{equation}
We define $\mathcal{Y}^{m+1}$ as the $\mathcal{Y}^m$-subalgebra of
$\Pi$ generated
by the elements $(f_{g,r}(a))^{\pm 1}$ obtained in this way, and
$G^{m+1}$ as the subgroup of invertible elements in
  $\mathcal{Y}^{m+1}$ generated by these $f_{g,r}(a)$ and by $G^m$.

\begin{example}\label{exsig} The
$q$-difference equation (\ref{firstqe}) is of the form
  \eqref{rell1-1}. Hence there is an element of $G^1$ representing its
  unique solution; namely, $\Sigma_{i,a}$ introduced in
    formula \eqref{totalSigma}.  Note that $\varpi_w(\Sigma^w_{i,a})$
    is the expansion of $(1-a_{\al_i}^{-1})^{-1}$ in negative powers
    of $a_{\al_i}$ if $w(\al_i) \in \Delta_+$, and
    $\varpi_w(\Sigma^w_{i,a})$ is the expansion of the same rational
    function $(1-a_{\al_i}^{-1})^{-1}$ in positive powers of
    $a_{\al_i}$ if $w(\al_i) \in \Delta_-$.
\end{example}

\begin{defi} The algebra $\ol{\mathcal{Y}}\subset \Pi$ (resp. the
  group $G\subset (\Pi)^\times$) is the union of the algebras
  $\mathcal{Y}^m$ (resp. of the groups $G^m$), $m\geq 0$ .
	\end{defi}
	
It follows from the definition that $\ol{\mathcal{Y}}$ contains the
diagonal subalgebra ${\mc Y} \subset \Pi$. Hence $\ol{\mathcal{Y}}$ is
a completion of $\mathcal{Y}$.

Thus, we have a sequence of embeddings
\begin{equation}    \label{embedd1}
{\mc Y} \subset \ol{\mc Y} \subset \Pi.
\end{equation}
Recall the homomorphism $\varpi: \Pi \to \pi$ introduced in Section
\ref{completion} and the ring ${\mc R}$ defined by formula \eqref{R}
which we have realized as a subring of $\pi$, see equation
\eqref{embedd}. The following result shows that the sequence
\eqref{embedd1} is mapped by $\varpi$ to the sequence \eqref{embedd}.

\begin{lem}\label{speo}
The restriction of the homomorphism
$\varpi: \Pi \to \pi$ to
$\ol{\mc Y}$ yields a homomorphism
\begin{equation}    \label{vart}
\ol{\mc Y} \to {\mc R}.
\end{equation}
\end{lem}

\begin{proof}
We need to show that for any $g\in \ol{\mathcal{Y}}$, there is an
element of ${\mc R}$ such that for any $w\in W$, $\varpi_w(E_w(g))$ is
the image of this element in $(\mathbb{Z}[y_i^{\pm 1}]_{i\in I})^w$
under the embedding \eqref{embw}.

We prove this statement for $g\in \mathcal{Y}^m$ by induction on
$m\geq 0$. This is clear on $\mathcal{Y}^0 = \mathcal{Y}$.  Suppose it
is true on $\mathcal{Y}^m$. Then it suffices to prove this statement
for the $f_{g,r}(a)$, where $g\in (G^m)^{adm}$,
$r\in\mathbb{Z}\setminus\{0\}$, $a\in\mathbb{C}^\times$. But for each
$w\in W$, we have
$$\varpi_w(E_w(f_{g,r}(a))) = (1 - \varpi_w(E_w(g)))^{-1},$$ which
is in ${\mc R}$ according to condition \eqref{adm}.  The statement of
the lemma follows.
\end{proof}

Next, we show that the action of $W$ on $\Pi$ preserves
$\ol{\mathcal{Y}}$ and the homomorphism \eqref{vart} commutes with the
action of $W$.

\begin{prop}     \label{pres}
  The operators $\Theta_i$ preserve the subalgebra
  $\ol{\mathcal{Y}}\subset \Pi$.
\end{prop}

\begin{proof} It suffices to prove by induction on $m\geq 0$ that
  $\Theta_i(G^m)\subset G^{m+1}$.  By definition, this is true for $m
  = 0$. Suppose that $\Theta_i(G^m)\subset G^{m+1}$. Since
  $s_i(\Delta) = \Delta$, we have $\Theta_i((G^m)^{adm})\subset
  (G^{m+1})^{adm}$.

Let $g\in (G^m)^{adm}$, $r\neq 0$ and $a\in\mathbb{C}^\times$. Then $\Theta_i(g)\in (G^{m+1})^{adm}$ and 
there is a unique $f_{\Theta_i(g),r}(a)\in \mathcal{Y}^{m+2}$ such that 
$$f_{\Theta_i(g),r}(a) = 1 + \tau_a(\Theta_i(g))f_{\Theta_i(g),r}(aq^{-r}).$$
The uniqueness implies that $\Theta_i(f_{g,r}(a)) = f_{\Theta_i(g),r}(a)$ and 
moreover $f_{\Theta_i(g),r}(a) \in G^{m+2}$. This concludes the proof.
\end{proof}

\begin{prop} The homomorphism $\ol{\mc Y} \to {\mc R}$ given by \eqref{vart}
commutes with the action of $W$.
\end{prop}

\begin{proof}
We have shown in Lemma \ref{speo} that the homomorphism $\varpi: \Pi
\to \pi$ restricts to a homomorphism $\ol{\mathcal{Y}}\rightarrow {\mc
  R}$. The former commutes with the action of $W$ by Lemma
\ref{Pipi}. Proposition \ref{pres} shows that $W$ preserves
$\ol{\mathcal{Y}} \subset \Pi$. Hence the result.
\end{proof}

This is consistent with the fact \cite{Fre} that $\varpi$
sends the subring of $q$-characters in $\Yim \subset \ol{\Yim}$ to the
subring of characters of finite-dimensional representations of $\g$ in
$\mathbb{Z}[y_i^{\pm 1}]_{i\in I} \subset \mathcal{R}$
and gives a precise sense in which the $W$-action on $\ol{\Yim}$
generated by the $\Theta_i$'s is a ``$q$-deformation'' of the
$W$-action on ${\mc R}$.

\begin{rem}    \label{sl2ex}
  Consider, for example, the homomorphism $\varpi:
  \ol{\mathcal{Y}} \rightarrow {\mc R}$ in the case of $sl_2$. Then we
  have $\varpi(\Sigma_{1,a}) = (1 - y^{-2})^{-1}$ and
$$\varpi(\Theta_1(\Sigma_{1,a})) = \varpi(1 -
\Sigma_{1,a}) = 1 - \frac{1}{1 - y^{-2}} = \frac{1}{1 - y^2} =
s_1\left(\frac{1}{1 - y^{-2}}\right) = s_1(\varpi(\Sigma_{1,a})).$$
\qed
\end{rem}

The action of the Weyl group on $\mathbb{Z}[y_i^{\pm 1}]_{i\in I}$ is
faithful. Hence we obtain the following.

\begin{cor} The $W$-action on $\Pi$ is faithful.
\end{cor}

\subsection{Braid group action}\label{bga}

Let $\overline{\mathcal{M}}$ be the subgroup of invertible elements of
the ring $\ol{\mathcal{Y}}$ generated by the multiplicative group
$\mathcal{M}$ of monomials in $\Yim$ (see Section \ref{Laurent}) and
by $G^m, m \geq 0$ (see Section \ref{comp2}). By construction,
$\Theta_i$ defines an automorphism of the group
$\overline{\mathcal{M}}$.

Introduce the truncation homomorphism $\Lambda :
\overline{\mathcal{M}}\rightarrow \mathcal{M}$ which assigns to $P
\in\overline{\mathcal{M}}$ the unique element $\Lambda(P) \in
\mathcal{M}$ such that $E_e(P) = \Lambda(P) \cdot \ol{P}$, where $\ol{P}
\in \ZZ[[A_{j,b}^{-1}]]_{j\in I, b\in\mathbb{C}^\times}$ has constant
term $\pm 1$. Thus, one can think of $\Lambda(P)$ as the leading monomial
of $E_e(P)$. The following result is obtained by a straightforward
calculation.

\begin{prop} The restriction of $\Lambda \circ \Theta_i$ to
  $\mathcal{M}$ is the Chari operator $T_i$ from \cite{C} given by
  formula \eqref{Ti}.
\end{prop}

\begin{rem}\label{crem} (i) More precisely, we obtain Chari's operators if we
  replace $q$ with $q^{-1}$.

(ii) It was shown in \cite{C} (see also \cite{BP}) that the operators
  $T_i$ generate an action of the braid group associated to $\g$, but
  not the Weyl group. Indeed, $T_i$ has
  infinite order. From the point of view of the above proposition, the
  reason is that $\Lambda$ and $\Theta_i$ do not commute.

(iii) In the $sl_2$ case, we have proved that $\chi_q(L(Y_{1,a}))$ is
  invariant under $\Theta_1$. But it is clearly not invariant under
  $T_1$ (moreover, $T_1(\chi_q(L))$ is not even in the image of
  $\chi_q$). In fact, it is easy to see that the subring of invariants
  of $T_i$ in $\Yim$ is trivial, i.e. is equal to $\Z$ (consists of
  the constant elements of $\Yim$).

(iv) It is proved in \cite{C} that if $w = s_{i_1}\cdots s_{i_N}$ is a
  reduced expression for $w\in W$ and $m'$ the lowest weight monomial 
	of a simple module $L(m)$, then $T_{i_1}\cdots T_{i_N}(m')$ is
  a monomial occurring in $\chi_q(L(m))$ with multiplicity $1$ (this is stated
	in \cite{C} with $m$ instead of $m'$ as $q^{-1}$ is used instead of $q$ in the definition of $T_i$). In
  fact, this exhausts all monomials in $\chi_q(L(m))$ of extremal
  weights (but in general there are many other monomials in
  $\chi_q(L(m))$).

(v) In a different setting, a possible relation between a Weyl group action and the Chari braid group action is discussed in \cite[Remark 4.16]{In} via tropicalization.  This could be related to the restriction of $\Lambda \circ \Theta_i$ in our Proposition.

(vi) We want to mention that in \cite{KKOP} an action of
  the braid group is constructed on a quantized Grothendieck ring of a
  certain category of finite-dimensional representations of
  $\U_q(\ghat)$. However, as far as we can see, it is unrelated to the
  actions we consider in this paper.
\end{rem}

\section{Expansions}\label{altpf}

 In this section we present explicit formulas for some elements
  of $\Pi$ related to the braid relations for Lie algebras of
  rank $2$. These formulas are not used in this paper, but they can be
  used to give a purely combinatorial proof of the
  braid relations, and we have also used them in \cite{FH3}.

\subsection{Type $A_2$}    \label{a2pf}
Here are explicit formulas for the
$w$-components $E_w(\Sigma_{ji,a}), w \in W$, of the solution of equation
(\ref{qdiffs1}) (which we have used in the alternative
  proof of the braid relations in the simply-laced case given in
  Section \ref{sltb}) in terms of the monomials $V_{k,a}^{(\alpha)}$
introduced in Definition \ref{TV}:
\begin{equation}\label{exp1}\begin{split}E_w(\Sigma_{ji,a}) 
= \begin{cases} \sum_{0\leq \beta\leq \alpha } V_{j,aq^{-1}}^{(\alpha)}V_{i,a}^{(\beta)}&\text{ if $w(\alpha_j)\in \Delta_+$ and $w(\alpha_i + \alpha_j)\in \Delta_+$,}
\\-\sum_{0\leq \beta, \alpha <\beta}
V_{j,aq^{-1}}^{(\alpha)}V_{i,a}^{(\beta)}&\text{ if $w(\alpha_j)\in \Delta_-$ and $w(\alpha_i + \alpha_j)\in \Delta_+$,}
\\-\sum_{\beta< 0,\beta\leq \alpha } V_{j,aq^{-1}}^{(\alpha)}V_{i,a}^{(\beta)}&\text{ if $w(\alpha_j)\in \Delta_+$ and $w(\alpha_i + \alpha_j)\in \Delta_-$,}
\\ \sum_{\alpha < \beta < 0} V_{j,aq^{-1}}^{(\alpha)}V_{i,a}^{(\beta)}&\text{ if $w(\alpha_j)\in \Delta_-$ and $w(\alpha_i + \alpha_j)\in \Delta_-$.}
\end{cases}\end{split}
\end{equation}
To prove the first formula, write $E_w(\Sigma_{ji,a})$ as a sum of
terms of increasing degrees in $A_{i,aq^{-2m}}^{-1}$ and compute the
corresponding coefficients by induction using the $q$-difference
equation (\ref{qdiffs1}).
One checks the other formulas in a similar way.

We can deduce formula \eqref{Sigmaji} directly from these expansions and hence obtain an alternative proof of the braid
  relations in the simply-laced case.

\subsection{Type $B_2$}\label{b2pf} Let us suppose that $C_{i,j} = -1$, $C_{j,i} =
-2$, $d_i = 2$ and $d_j = 1$. As in Section \ref{a2pf}, one
finds, using the relevant $q$-differences equations, the following
explicit formulas:
\begin{equation}\label{exp2}E_e(\Sigma_{ij,a}) = \sum_{0\leq \beta\leq
    \text{Min}(2\alpha,2\alpha'+1)}
  V_{i,aq^{-2}}^{(\alpha)}V_{i,aq^{-4}}^{(\alpha')}V_{j,a}^{(\beta)},\end{equation}
\begin{equation}\label{exp5}E_e(\Sigma_{iji,a}) = \sum_{0\leq \gamma\leq \beta/2 \leq
    \alpha}
  V_{i,aq^{-2}}^{(\alpha)}V_{j,a}^{(\beta)}V_{i,a}^{(\gamma)},
\end{equation}
\begin{equation}\label{exp3}E_e(\Sigma_{ji,a})  = \sum_{0\leq 2\beta\leq \alpha} V_{j,a}^{(\alpha)}V_{i,a}^{(\beta)},\end{equation}
\begin{multline}\label{exp4}E_e(\Sigma_{jij,a}) =
  \sum_{0\leq \beta\leq \alpha/2, 0\leq \beta'\leq (\alpha+1)/2, 0\leq
    \gamma\leq \text{Min}(1 + 2\beta, 2\beta')}
  V_{j,aq^{-4}}^{(\alpha)}V_{i,aq^{-4}}^{(\beta)}V_{i,aq^{-2}}^{(\beta')}V_{j,a}^{(\gamma)}.
\end{multline}

\subsection{Type $G_2$}
Let us suppose that $C_{i,j} = -1$, $C_{j,i} = -3$, $d_i = 3$ and $d_j
= 1$. We have 
\begin{equation}\label{gexpji}E_e(\Sigma_{ji,a}) = \sum_{0\leq 3\epsilon\leq \delta} V_{j,aq}^{(\delta)}V_{i,a}^{(\epsilon)},\end{equation}
\begin{equation}\label{gexpij}E_e(\Sigma_{ij,a}) = \sum_{0\leq \epsilon\leq \text{Min}(3\delta_1, 3\delta_2 + 1 , 3\delta_3 + 2)} V_{i,aq^{-3}}^{(\delta_1)}V_{i,aq^{-5}}^{(\delta_2)}
V_{i,aq^{-7}}^{(\delta_3)}V_{j,a}^{(\epsilon)}.\end{equation}
 By a similar argument, we obtain the following expansions
  (with appropriate conditions on the domains of summation):
$$E_e(\Sigma_{iji,a}) = \sum V_{i,aq^{-2}}^{(\gamma)}V_{i,aq^{-4}}^{(\gamma')}V_{j,aq}^{(\delta)}V_{i,a}^{(\epsilon)},$$
$$E_e(\Sigma_{jiji,a}) = \sum V_{j,aq^{-3}}^{(\beta)}V_{i,aq^{-2}}^{(\gamma)}V_{i,aq^{-4}}^{(\gamma')}V_{j,aq}^{(\delta)}V_{i,a}^{(\epsilon)},$$
$$E_e(\Sigma_{ijiji,a}) = \sum V_{i,aq^{-6}}^{(\alpha)}V_{j,aq^{-3}}^{(\beta)}V_{i,aq^{-2}}^{(\gamma)}V_{i,aq^{-4}}^{(\gamma')}V_{j,aq}^{(\delta)}V_{i,a}^{(\epsilon)},$$
$$E_e(\Sigma_{jij,a}) = \sum V_{j,aq^{-4}}^{(\gamma_1)}V_{j,aq^{-6}}^{(\gamma_2)}V_{i,aq^{-3}}^{(\delta_1)}V_{i,aq^{-5}}^{(\delta_2)}V_{i,aq^{-7}}^{(\delta_3)}V_{j,a}^{(\epsilon)},$$
$$E_e(\Sigma_{ijij,a}) = \sum V_{i,aq^{-7}}^{(\beta_1)}V_{i,aq^{-9}}^{(\beta_2)}V_{i,aq^{-11}}^{(\beta_3)}V_{j,aq^{-4}}^{(\gamma_1)}V_{j,aq^{-6}}^{(\gamma_2)}
V_{i,aq^{-3}}^{(\delta_1)}V_{i,aq^{-5}}^{(\delta_2)}V_{i,aq^{-7}}^{(\delta_3)}V_{j,a}^{(\epsilon)},$$
$$E_e(\Sigma_{jijij,a}) = \sum V_{j,aq^{-10}}^{(\alpha)}V_{i,aq^{-7}}^{(\beta_1)}V_{i,aq^{-9}}^{(\beta_2)}V_{i,aq^{-11}}^{(\beta_3)}V_{j,aq^{-4}}^{(\gamma_1)}
V_{j,aq^{-6}}^{(\gamma_2)}V_{i,aq^{-3}}^{(\delta_1)}V_{i,aq^{-5}}^{(\delta_2)}V_{i,aq^{-7}}^{(\delta_3)}V_{j,a}^{(\epsilon)}.$$

 Using the above expansions, one can give a purely
  combinatorial proof of the braid relations for $B_2$ and
  $G_2$. However, since we have already given a uniform proof of these
  relations in Section \ref{gtb}, we will omit the details.


\begin{thebibliography}{KKOP}

\bibitem[Be]{bec} {J. Beck},
\newblock {\em Braid group action and quantum affine algebras}, 
\newblock Comm. Math. Phys. {\bf 165} 
(1994) 555--568.

\bibitem[BP]{BP} P. Bouwknegt and K. Pilch, {\em On deformed
  W-algebras and quantum affine algebras},
  Adv. Theor. Math. Phys. {\bf 2} (1998) 357--397.

\bibitem[C]{C} V. Chari, 
{\em Braid group actions and tensor products}, 
Int. Math. Res. Not. {\bf 2002} 357--382.

\bibitem[CH]{CH} V. Chari and D. Hernandez,
{\em Beyond Kirillov-Reshetikhin modules}, 
in Quantum affine algebras, extended affine Lie algebras, 
and their applications,  
Contemp. Math. {\bf 506}, pp. 49--81, AMS Providence, 2010.

\bibitem[CP]{CP} {V. Chari and A. Pressley},
{\em A guide to quantum groups},
Cambridge University Press, 1994.

\bibitem[Da]{da} I. Damiani, {\it From the Drinfeld realization to the
  Drinfeld--Jimbo presentation of affine quantum algebras:
  Injectivity}, Publ. Res. Inst. Math. Sci. {\bf 51} (2015) 131--171.

\bibitem[Dr]{Dri2} {V. Drinfeld}, 
\newblock {\em A new realization of Yangians and of 
quantum affine algebras}, 
\newblock Soviet Math. Dokl. {\bf 36}
(1988) 212--216.  

\bibitem[ESV]{ESV} S. Ekhammar, H. Shu, and D. Volin, {\em Extended
  systems of Baxter Q-functions and fused flags I: simply-laced case},
Preprint arXiv:2008.10597.

\bibitem[EV]{EV} S. Ekhammar and D. Volin, {\em Bethe Algebra
  using Pure Spinors}, Preprint arXiv:2104.04539.

\bibitem[FH1]{FH} {E. Frenkel and D. Hernandez}, \newblock {\em
  Baxter's Relations and Spectra of Quantum Integrable Models},
  \newblock Duke Math. J. {\bf 164} (2015) 2407--2460.

\bibitem[FH2]{FH2} {E. Frenkel and D. Hernandez}, \newblock {\em
  Spectra of quantum KdV Hamiltonians, Langlands duality, and affine
  opers}, \newblock Comm. Math. Phys. {\bf 362} (2018)
  361--414.

\bibitem[FH3]{FH3} E. Frenkel and D. Hernandez, {\em Extended Baxter
  Relations and QQ-Systems for Quantum Affine Algebras},
Comm. Math. Phys. {\bf 405} (2024) 190 (arXiv:2312.13256).

\bibitem[FM]{Fre2} {E. Frenkel and E. Mukhin}, 
\newblock {\em Combinatorics of $q$-characters of finite-dimensional 
representations of quantum affine algebras}, 
\newblock Comm. Math. Phys. {\bf 216}
(2001) 23--57.

\bibitem[FR]{Fre} {E. Frenkel and N. Reshetikhin}, 
\newblock {\em The $q$-characters of representations of quantum 
affine algebras and deformations of $W$-Algebras}, 
\newblock in Recent Developments in Quantum Affine Algebras and 
related topics, 
\newblock 
 Contemp. Math. {\bf 248} (1999) 163--205 (arXiv:math/9810055).

\bibitem[GHL]{GHL} {C. Geiss, D. Hernandez, and B. Leclerc}, 
\newblock {\em Representations of shifted quantum affine algebras and cluster algebras I. The simply-laced case},
\newblock Proc. Lond. Math. Soc. (3) {\bf 129} (2024) e12630
(arXiv:2401.04616).
	
\bibitem[H]{H} D. Hernandez, {\em Algebraic approach to $q,t$-characters}, 
Adv. Math. {\bf 187} (2004) 1--52.

\bibitem[HJ]{HJ} {D. Hernandez and M. Jimbo}, \newblock {\em
  Asymptotic representations and Drinfeld rational fractions},
  \newblock Compos. Math. {\bf 148} (2012) 1593--1623.

\bibitem[In]{In} R. Inoue, {\em Cluster realizations of Weyl groups
  and $q$-characters of quantum affine algebras}, 
	\newblock Lett. Math. Phys. {\bf 111} (2021) 4 (arXiv:2003.04491).

\bibitem[IY]{IY} R. Inoue and T. Yamazaki, {\em Invariants of Weyl
  group action and $q$-characters of quantum affine algebras}, 
\newblock Algebras and Rep. Theory {\bf 26} (2023) 3167--3183
(arXiv:2207.09867).

\bibitem[Kac]{ka}{V. Kac},
\newblock {\em Infinite Dimensional Lie Algebras}, 3rd Edition,
\newblock Cambridge University Press, Cambridge, 1990.

\bibitem[KKOP]{KKOP} M. Kashiwara, M. Kim, S-J. Oh, and E. Park, {\em
  Braid group action on the module category of quantum affine
  algebras}, Proc. Japan Acad. Ser. A Math. Sci. {\bf 97} (2021)
  13--18 (arXiv:2004.04939).

\bibitem[L]{L} B. Leclerc,
{\em Quantum loop algebras, quiver varieties, and cluster algebras},
in Representations of Algebras and Related Topics, (A. Skowro\'nski
and K. Yamagata, eds.), pp. 117--152,
European Math. Soc. Series of Congress Reports, 2011.

\bibitem[MRV1]{MRV1} D. Masoero, A. Raimondo, and D. Valeri, {\em Bethe
  Ansatz and the Spectral Theory of affine Lie algebra-valued
  connections. The simply-laced case}, Comm. Math. Phys. {\bf 344}
  (2016) 719--750.
  
\bibitem[MRV2]{MRV2} {D. Masoero, A. Raimondo, and D. Valeri},
\newblock {\em Bethe Ansatz and the Spectral Theory of affine Lie
algebra-valued connections. The non simply-laced case},
\newblock Comm. Math. Phys. {\bf 349} (2017) 1063--1105.

\bibitem[MV1]{MV1} E. Mukhin and A. Varchenko, {\em Populations of
  solutions of the XXX Bethe equations associated to Kac--Moody
  algebras}, in Infinite-dimensional Aspects of Representation Theory
  and Applications, Contemp. Math. {\bf 392}, pp. 95--102, AMS, 2005.

\bibitem[MV2]{MV2} E. Mukhin and A. Varchenko, {\em Discrete
  Miura Opers and Solutions of the Bethe Ansatz Equations},
  Commun. Math. Phys. {\bf 256} (2005) 565--588.

\end{thebibliography}
\end{document}